\documentclass[11pt]{article}
\usepackage[utf8]{inputenc}
\usepackage[margin=1.0in]{geometry}
\usepackage{mathtools} 
\usepackage{amsthm}
\usepackage{amssymb} 
\usepackage{amsfonts, dsfont} 
\usepackage{graphicx} 
\usepackage{color} 
\usepackage{subcaption}
\usepackage{enumerate} 
\usepackage{fancybox} 
\usepackage{float}  
\usepackage{verbatim}
\usepackage{stmaryrd}
\usepackage{empheq}
\usepackage{booktabs}
\usepackage{mathrsfs}
\usepackage{tikz}
\usepackage{ulem}
\usepackage{appendix}
\usepackage{hyperref}

\newtheorem{theorem}{Theorem}[section]
\newtheorem{prop}{Proposition}[section]
\newtheorem{lemma}{Lemma}[section]
\newtheorem{remark}{Remark}[section]

\newtheorem{definition}{Definition}[section]
\newtheorem{note}{Note}

\newcommand{\R}{\mathbb{R}}
\newcommand{\Z}{\mathbb{Z}}
\newcommand{\N}{\mathbb{N}}
\newcommand{\C}{\mathbb{C}}
\newcommand{\dd}{\; \mathrm{d}}

\begin{document}
\title{The Goursat problem at the horizons for the Klein-Gordon equation on the De Sitter-Kerr metric}
\author{Pascal Millet}
\date{2019}
\maketitle

\begin{abstract}
The main topic is the Goursat problem at the horizons for the Klein-Gordon equation on the De Sitter-Kerr metric when the angular momentum per unit of mass of the black hole is small. We solve the Goursat problem for fixed angular momentum $n$ of the field (with the restriction that $n\neq 0$ in the case of a massless field).
\end{abstract}

\setcounter{tocdepth}{3}
\tableofcontents

\section{Introduction}
There has been a lot of activity concerning scattering theory for hyperbolic equations on black hole type spacetimes over the last years. There are two different ways to formulate what is called {\it asymptotic completeness} on these spacetimes. One formulation is in terms of wave operators which make the link between the dynamics one wants to study and a simplified dynamics. The asymptotic completeness can then be understood as saying that the long time dynamics of the complete system is well described by the simplified dynamics. In the massless case if one chooses as simplified dynamics a dynamics linked to transport along certain null geodesics in the given spacetime, then asymptotic completeness can be understood as an existence and uniqueness result in energy spaces for a characteristic Cauchy problem at infinity, see \cite{HaNi} or \cite{Ni} for details. The precise understanding of scattering properties of fields on a black hole spacetime is crucial to define quantum states on the given spacetime or to describe the Hawking effect in a rigorous way, see \cite{DaMoPi}, \cite{Ba}, \cite{Ha09}.  

The most important of these black hole spacetimes is the (De Sitter) Kerr spacetime which describes rotating black holes. The first asymptotic completeness result for a hyperbolic equation on the Kerr spacetime was obtained by H\"afner for non superradiant modes of the Klein-Gordon equation (see \cite{Ha03}), later Nicolas and H\"afner proved asymptotic completeness for the Dirac equation on the Kerr spacetime, see \cite{HaNi}. Asymptotic completeness for the Klein-Gordon equation on the De Sitter-Kerr black hole for fixed angular momentum of the field was established by G\'erard, Georgescu and H\"afner (see \cite{GGH}) and for the wave equation on the Kerr black hole by Dafermos-Rodnianski-Shlapentokh Rothman (without restriction of the angular momentum of the field), see \cite{DRSR}. The main difference between the Dirac and the wave or Klein-Gordon equation is the existence of superradiance for the latter, meaning that there doesn't exist any positive conserved quantity. Note that superradiance is also present for the charged Klein-Gordon equation on the De Sitter-Reissner-Nordstr\"om spacetime, for which asymptotic completeness has been obtained recently by Besset when the charge product (charge of the black hole and charge of the field) is small with respect to the mass of the field, see \cite{Be}. When the mass of the field is small with respect to the charge product, exponentially growing finite energy solutions can exist and asymptotic completeness does not hold in that case, see \cite{besset2020existence} for details.

The results of \cite{GGH} are formulated in terms of wave operators and the aim of the present paper is to give a geometric interpretation of the results in \cite{GGH} in terms of an existence and uniqueness result in energy spaces for the characteristic Cauchy problem at infinity. Because of the existence of two horizons the result holds for the Klein-Gordon equation. A similar result would certainly not hold for the Klein-Gordon equation on the Kerr spacetime as part of the energy could escape to future timelike infinity.  
\subsection*{Acknowledgment.}
I am very grateful to my phD advisor Dietrich Häfner for numerous discussions which have contributed a lot to the present work. I also thank Nicolas Besset for punctual but useful exchanges. 

\section{Main result}
\subsection{Kerr De Sitter space-time}
We recall the expression of the De Sitter-Kerr metric in Boyer-Lindquist coordinates for a cosmological constant $\Lambda>0$, a mass parameter $M>0$ and an angular momentum per unit of mass $a$:
\begin{align}
g := & \frac{\Delta_r-a^2\sin^2(\theta)\Delta_{\theta}}{\lambda^2 \rho^2} \dd t^2 + \frac{2a \sin^2(\theta)((r^2+a^2)\Delta_{\theta} -\Delta_r)}{\lambda^2 \rho^2}\dd t \dd \phi \notag \\
& - \frac{\rho^2}{\Delta_r}\dd r^2 - \frac{\rho^2}{\Delta_{\theta}}\dd \theta^2 - \frac{\sin^2(\theta)\sigma^2}{\lambda^2\rho^2}\dd \phi^2 \label{metric}
\end{align}
with the notations
\begin{align*}
\rho^2 &:= r^2 + a^2 \cos^2 \theta \\
\Delta_r &:= (1-\frac{1}{3}\Lambda r^2)(r^2 + a^2)-2Mr \\
\Delta_{\theta} &:= 1 + \frac{1}{3}\Lambda a^2 \cos^2 \theta \\
\sigma^2 &:= (r^2 + a^2)^2\Delta_{\theta} - a^2 \Delta_r \sin^2 \theta \\
\lambda &:= 1 + \frac{1}{3}\Lambda a^2
\end{align*}
We also name each coefficient:
\[
g = g_{t,t}\dd t^2 + 2g_{t, \phi}\dd t \dd \phi+ g_{r,r}\dd r^2 + g_{\theta, \theta} \dd \theta^2 + g_{\phi, \phi} \dd \phi^2
\]
We assume that $\Delta_r>0$ for $r\in(r_-, r_+)$ where $r_-$ and $r_+$ are positive simple roots. (For $a = 0$, this is true when $9\Lambda M^2 < 1$; it remains true if $a$ is small enough.) As a consequence, if we define $\alpha(r):= \frac{\Delta_r}{(r-r_-)(r_+-r)}$, we get that $\alpha(r)^{\pm 1} \in C^{\infty}([r_-, r_+], (0, +\infty))$. 
The manifold $\mathcal{M}=\mathbb{R}_t\times(r_-, r_+)\times \mathbb{S}^2$ endowed with the metric \eqref{metric} is what we call the Kerr De Sitter space-time. It describes the space-time outside the horizon of a spinning black-hole in an expanding universe ($\Lambda>0$).
 We introduce a Regge Wheeler coordinate defined (up to a constant that we fix arbitrarily) by the condition 
\begin{equation}
\frac{dx}{dr} = \lambda \frac{r^2+a^2}{\Delta_r}
\end{equation}

\begin{note}
Note that $x\rightarrow \pm \infty$ when $r\rightarrow r_{\pm}$. Moreover, we have (see for example the beginning of section (9.1) of \cite{GGH})
\[
|r_+-r| \leq e^{-\kappa_+ x}
\]
and 
\[
|r - r_- | \leq e^{\kappa_- x}
\]
Where $\kappa_+>0$ and $\kappa_->0$ are the surface gravities of the horizons.
\end{note}

We normalize the principal null vector fields (see \cite{KDS} equation (25) and \cite{borthwick} in subsection 4.1) so that they are compatible with the time foliation. 
\begin{align*}
v_{+}&= \partial_t + \partial_x + \frac{a}{a^{2}+r^{2}}\partial_\phi \\
v_{-}&= \partial_t - \partial_x + \frac{a}{a^{2}+r^{2}}\partial_\phi
\end{align*}
The main results that we use from \cite{GGH} are stated for a fixed angular momentum with respect to the axis of rotation of the black hole. In what follows, we fix the mode $n\in \Z$ and we consider the operators induced on $\mathcal{Y}^n:=ker (D_{\phi} - n)\subset L^2(\R \times \mathbb{S}^2)$, ($D_{\phi}$ is considered as an unbounded operator on $L^2(\R \times \mathbb{S}^2)$ with domain $\left\{ u\in L^2(\R\times \mathbb{S}^2): D_{\phi}u\in L^2(\R\times\mathbb{S}^2)\right\}$ ).
We also define $Y^n := ker(D_{\phi}-n)\subset L^2(\mathbb{S}^2)$ where $D_{\phi}$ is considered as an operator on $L^2(\mathbb{S}^2)$. Remark that $Y^n$ and $\mathcal{Y}^n$ endowed with the natural $L^2$ norm are separable Hilbert spaces.

We introduce the $n$ dependent differential operator which coincide with the spatial part of the vector fields $v_+$ and $v_-$ on $\mathcal{Y}^n$:
\begin{align*}
w_{+}&=\partial_x + \frac{ian}{a^{2} + r^{2}} \\
w_{-}&=-\partial_x + \frac{ian}{a^{2}+ r^{2}}
\end{align*}

\label{Constr}
We now define the horizons. $\mathcal{M}$ is not maximally extended. In other words we can find a Lorentzian manifold $\mathcal{M}'$ solution to the vacuum Einstein equation with cosmological constant $\Lambda$ such that $\mathcal{M}$ is isometric to a strict submanifold of $\mathcal{M}'$.
\begin{definition}
We fix once and for all $r_0\in (r_-,r_+)$.
We define two functions of $r$:
\begin{align*}
T(r) &= \int_{r_0}^r\frac{\lambda (a^2+r^2)}{\Delta_r} \dd r \\
A(r) &= \int_{r_0}^r\frac{\lambda a}{\Delta_r} \dd r
\end{align*}
\end{definition}
\begin{remark}
 $T$ and $A$ are increasing homeomorphisms between $(r_-, r_+)$ and $\mathbb{R}$.
\end{remark}

\begin{definition}
We define two changes of coordinates 
\begin{align*}
\Psi^{out}(t,r,\theta, \phi) &= (t-T(r), r, \theta, \phi - A(r)) =: ({}^*t, r, \theta, {}^*\phi) \text{   ${}^*$Kerr coordinates}\\
\Psi^{in}(t,r,\theta, \phi) &= (t+T(r), r, \theta, \phi + A(r)) =: (t^*, r, \theta, \phi^*) \text{   Kerr${}^*$ coordinates}
\end{align*}
\end{definition}
\begin{remark}
Note that the $\phi$ coordinates is considered as valued in $\mathbb{S}^1$, so for $k\in \mathbb{Z}$, $\phi = \phi + 2k\pi$
\end{remark}
In terms of these new coordinates, we can extend analytically the metric to a bigger manifold (to see this construction in more details, we refer to \cite{borthwick} section 4). In particular we can add the following null hypersurfaces (called horizons) to the space-time:
\begin{align}
\mathfrak{H}^{future}_+ &:= \mathbb{R}_{{}^*t}\times {r_+}\times \mathbb{S}^2\\
\mathfrak{H}^{past}_+ &:= \mathbb{R}_{t^*}\times {r_+}\times \mathbb{S}^2\\
\mathfrak{H}^{future}_- &:= \mathbb{R}_{t^*}\times {r_-}\times \mathbb{S}^2 \\
\mathfrak{H}^{past}_- &:= \mathbb{R}_{{}^*t}\times {r_-}\times \mathbb{S}^2
\end{align}
$\mathfrak{H}^{future}_-$ (resp. $\mathfrak{H}^{past}_-$) is the future (resp. past) horizon of the black hole. $\mathfrak{H}^{future}_+$ (resp. $\mathfrak{H}^{past}$) is the future (resp. past) Cauchy horizon. 
We also define $\bar{\mathcal{M}}:= \mathcal{M}\cup \mathfrak{H}^{future}_+ \cup \mathfrak{H}^{past}_+ \cup \mathfrak{H}^{future}_- \cup \mathfrak{H}^{past}_- $. We refer to \cite{borthwick} for the construction of the maximal extension of Kerr De Sitter black hole.

\begin{figure}[!h]
\centering
\includegraphics[scale = 0.15]{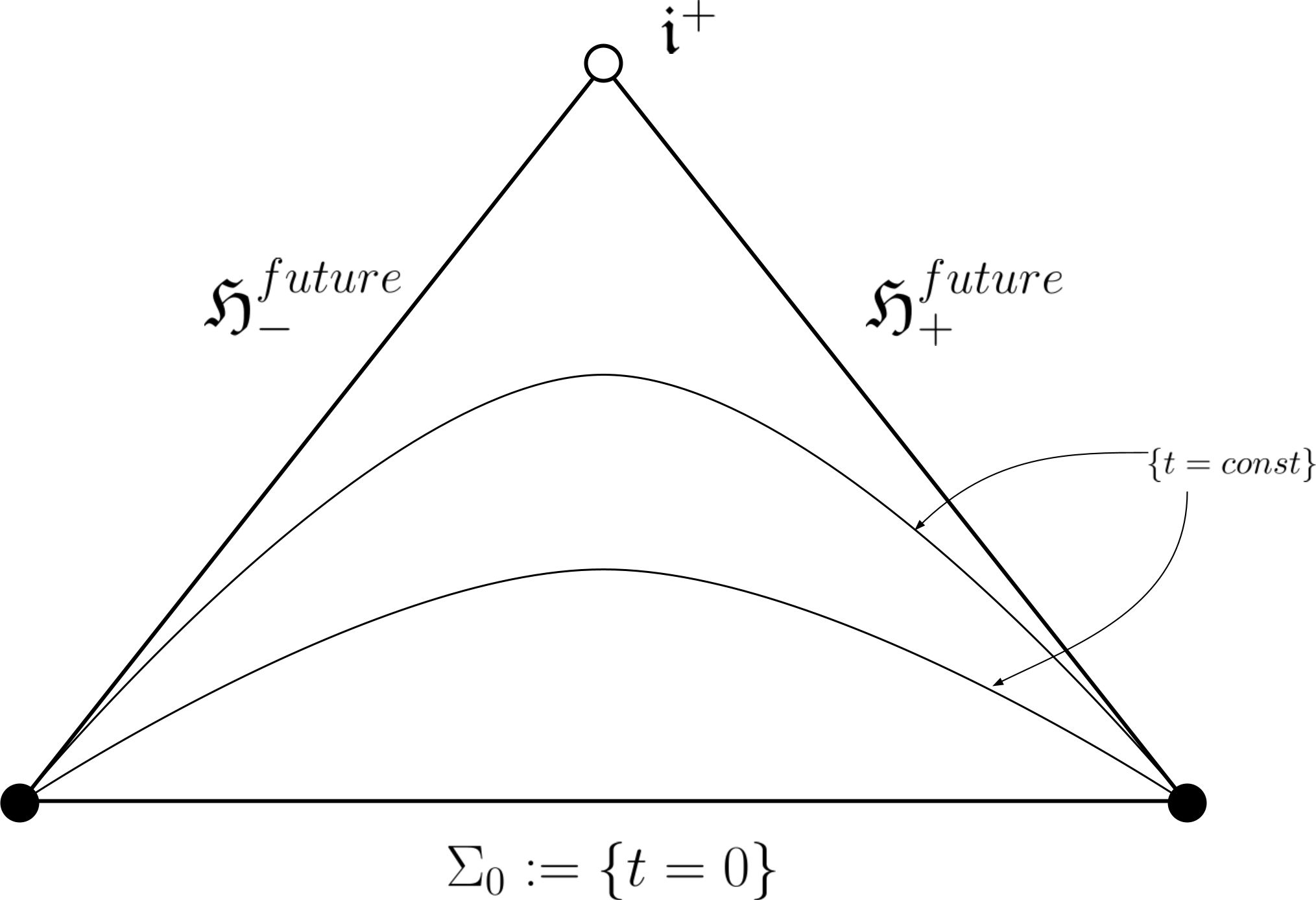}
\caption{\label{penroseDiagram} Carter-Penrose diagram of the part $\left\{t\geq 0\right\}$ of the Kerr De Sitter space-time \\
$\mathfrak{i}^+$ corresponds to the future timelike infinity (it is not a point of the spacetime)}
\end{figure}

\subsection{The main theorem}

The main result of the present paper is a formulation of asymptotic completeness for the Klein-Gordon equation $(\square + m^2)u=0$ on the Kerr De Sitter space-time (for sufficiently small angular momentum $a$ of the black hole) in term of a characteristic Cauchy problem on the future horizons.

\label{mainRes}
We consider initial data given as elements in some energy space over the surface $\Sigma_0:=\left\{t=0\right\}$ (diffeomorph to $(r_-, r_+)\times \mathbb{S}^2$). As in \cite{GGH}, we have to consider initial data in the kernel of $D_\phi - n$ where $D_\phi = \frac{1}{i}\partial_\phi$ and $n\in \mathbb{Z}$. Keep in mind that our results are not uniform with respect to $n$.
To define the energy, we introduce the stress-energy tensor associated with the Klein Gordon equation:
\[
T(u)_{a,b} = \nabla_a u \nabla_b u - \frac{1}{2}g_{a,b}\left<\nabla u, \nabla u\right> + \frac{1}{2}g_{a,b}m^2u^2
\]
It is divergence free if $u$ is a solution of the Klein-Gordon equation.The usual energy current one-form is obtained by contracting $T(u)$ with $\partial_t$ (Killing vector field). However, because $\partial_t$ is not globally timelike, the flux of this energy form through $\Sigma_0$ can be negative. Therefore, we replace $\partial_t$ by $X := \frac{\nabla t}{g(\nabla t, \nabla t)}$, which is timelike and continuous on $\bar{\mathcal{M}}$ (but not Killing). The flux of this modified energy form through $\Sigma_0$ is used to define an energy space $\dot{\mathcal{E}}^n$ (see section \ref{energySpace} for an explicit definition). Similarly the flux of the energy form through the future horizons gives us energy spaces $\mathcal{E}^n_{\mathfrak{H}_+}$ and $\mathcal{E}^n_{\mathfrak{H}_-}$ (see section \ref{energyHorizon} for an explicit definition).
If the initial Cauchy data $u \in \dot{\mathcal{E}}^n$ are smooth and compactly supported, Leray's theorem gives us the existence and uniqueness of a smooth (on $\bar{\mathcal{M}}$) solution to the Klein-Gordon equation. We can define its trace on $\mathfrak{H}^{future}_\pm$ : $\mathcal{T}_\pm u\in C^{\infty}(\mathfrak{H}^{future}_\pm)$. Our main theorem is the following:
\begin{theorem}
We assume that $n\neq 0$ or $m^2>0$.
There exists $C>0$ such that for all $u\in (C^{\infty}_0(\Sigma_0)\cap \mathcal{Y}^n)^2$,
\[
\left\|u\right\|_{\dot{\mathcal{E}}^n}\leq C\left(\left\|\mathcal{T}_-u\right\|_{\mathcal{E}^n_{\mathfrak{H}_-}}+\left\|\mathcal{T}_+u\right\|_{\mathcal{E}^n_{\mathfrak{H}_+}}\right)
\]
Then the application
\[
\begin{cases}
(C^{\infty}_0(\Sigma_0)\cap \mathcal{Y}^n)^2 \rightarrow C^{\infty}(\mathfrak{H}^{future}_-)\times C^{\infty}(\mathfrak{H}^{future}_+) \\
u \mapsto (\mathcal{T}_- u, \mathcal{T}_+ u)
\end{cases}
\]

has a unique continuous extension from $\dot{\mathcal{E}}^n$ to $\mathcal{E}^n_{\mathfrak{H}_-} \oplus \mathcal{E}^n_{\mathfrak{H}_+}$. Moreover this extension is a homeomorphism.
\end{theorem}

We will even be more precise and give an explicit description of the trace operator in terms of wave operators (see theorem \ref{mainTheo}).

\subsection{Organization of the paper}
The main theorem is a consequence of the description in terms of wave operators associated with well chosen comparison dynamics. In section \ref{FuncSet}, we introduce the functional setting, in particular the Klein-Gordon dynamics, the comparison dynamics and a Kirchhoff formula. We construct two comparison dynamics (one for each horizon) and therefore two wave operators $W_{T,+}$ and $W_{T,-}$ and two corresponding inverse wave operators $\Omega_{T,+}$ and $\Omega_{T,-}$. We then glue them together to obtain a global wave operator $W$ and a global inverse wave operator $\Omega$. In section \ref{properties}, we prove that the two global operators are inverse to each other. Finally, we make a connection between the trace operator and the global inverse wave operator and we prove the main theorem \ref{mainTheo}.

\section{Functional setting}
\label{FuncSet}
\subsection{Notations for the Klein-Gordon dynamics}
We recall the notations introduced in the sections 11 and 12 of the article \cite{GGH}. 
We are interested in the Klein-Gordon equation on the De Sitter-Kerr space-time $(\square_g + m^2)u = 0$.
We rewrite it in Boyer-Lindquist coordinates, after multiplication by $\frac{\rho^2\Delta_r\Delta_\theta}{\lambda^2\sigma^2}$ (to have the coefficient in front of $\partial_t^2$ equal to 1):
\begin{align*}
\left(\vphantom{\partial_t^2 - 2\frac{a(\Delta_r - (r^2 + a^2)\Delta_\theta)}{\sigma^2}\partial_\phi\partial_t - \frac{(\Delta_r - a^2\sin^2 \theta \Delta_\theta)}{\sin^2\theta \sigma^2}\partial_\phi^2-\frac{\Delta_r\Delta_\theta}{\lambda^2\sigma^2}\partial_r\Delta_r\partial_r - \frac{\Delta_r\Delta_\theta}{\lambda^2\sin\theta\sigma^2}\partial_\theta \sin\theta\partial_\theta+ \frac{\rho^2\Delta_r\Delta_\theta}{\lambda^2\sigma^2}m^2}\right. \partial_t^2 - 2\frac{a(\Delta_r - (r^2 + a^2)\Delta_\theta)}{\sigma^2}\partial_\phi\partial_t - \frac{(\Delta_r - a^2\sin^2 \theta \Delta_\theta)}{\sin^2\theta \sigma^2}\partial_\phi^2&\notag\\
-\frac{\Delta_r\Delta_\theta}{\lambda^2\sigma^2}\partial_r\Delta_r\partial_r - \frac{\Delta_r\Delta_\theta}{\lambda^2\sin\theta\sigma^2}\partial_\theta \sin\theta \Delta_\theta\partial_\theta+ \frac{\rho^2\Delta_r\Delta_\theta}{\lambda^2\sigma^2}m^2&\left. \vphantom{\partial_t^2 - 2\frac{a(\partial_r - (r^2 + a^2)\Delta_\theta)}{\sigma^2}\partial_\phi\partial_t - \frac{(\Delta_r - a^2\sin^2 \theta \Delta_\theta)}{\sin^2\theta \sigma^2}\partial_\phi^2
-\frac{\Delta_r\Delta_\theta}{\lambda^2\sigma^2}\partial_r\Delta_r\partial_r - \frac{\Delta_r\Delta_\theta}{\lambda^2\sin\theta\sigma^2}\partial_\theta \sin\theta\partial_\theta+ \frac{\rho^2\Delta_r\Delta_\theta}{\lambda^2\sigma^2}m^2}\right)u= 0
\end{align*}
To simplify the analysis, we use the unitary transform 
\[U:L^2\left(\Sigma_0, \frac{\sigma^2}{\Delta_r\Delta_\theta}\dd r \dd \omega \right) \rightarrow L^2(\Sigma_0, \dd x \dd \omega), \text{     }u\mapsto \frac{\sigma}{\sqrt{\lambda(r^2+a^2)\Delta_\theta}}u
\]
We write the equation on $v= Uu$:
\begin{align}
\left(\vphantom{\partial_t^2 - 2\frac{a(\Delta_r - (r^2 + a^2)\Delta_\theta)}{\sigma^2}\partial_\phi\partial_t - \frac{\Delta_r - a^2\sin^2\theta}{\sin^2\theta \sigma^2}\partial_\phi^2 -\frac{\sqrt{(r^2+a^2)\Delta_\theta}}{\sigma}\partial_x(r^2+a^2)\partial_x\frac{\sqrt{(r^2+a^2)\Delta_\theta}}{\sigma}-\frac{\sqrt{\Delta_r\Delta_\theta}}{\lambda\sin\theta\sigma}\partial_\theta \Delta_\theta\partial_\theta \frac{\sqrt{\Delta_r\Delta_\theta}}{\lambda\sigma} + \frac{\rho^2\Delta_r\Delta_\theta}{\lambda^2\sigma^2}m^2}\right.\partial_t^2 - 2\frac{a(\Delta_r - (r^2 + a^2)\Delta_\theta)}{\sigma^2}\partial_\phi\partial_t - \frac{\Delta_r - a^2\sin^2\theta\Delta_\theta}{\sin^2\theta \sigma^2}\partial_\phi^2 &\notag\\
-\frac{\sqrt{(r^2+a^2)\Delta_\theta}}{\sigma}\partial_x(r^2+a^2)\partial_x\frac{\sqrt{(r^2+a^2)\Delta_\theta}}{\sigma}& \notag\\
-\frac{\sqrt{\Delta_r\Delta_\theta}}{\lambda\sin\theta\sigma}\partial_\theta \sin\theta \Delta_\theta\partial_\theta \frac{\sqrt{\Delta_r\Delta_\theta}}{\lambda\sigma} + \frac{\rho^2\Delta_r\Delta_\theta}{\lambda^2\sigma^2}m^2&\left. \vphantom{\partial_t^2 - 2\frac{a(\Delta_r - (r^2 + a^2)\Delta_\theta)}{\sigma^2}\partial_\phi\partial_t - \frac{\Delta_r - a^2\sin^2\theta}{\sin^2\theta \sigma^2}\partial_\phi^2 -\frac{\sqrt{(r^2+a^2)\Delta_\theta}}{\sigma}\partial_x(r^2+a^2)\partial_x\frac{\sqrt{(r^2+a^2)\Delta_\theta}}{\sigma}-\frac{\sqrt{\Delta_r\Delta_\theta}}{\lambda\sin\theta\sigma}\partial_\theta \Delta_\theta\partial_\theta \frac{\sqrt{\Delta_r\Delta_\theta}}{\lambda\sigma} + \frac{\rho^2\Delta_r\Delta_\theta}{\lambda^2\sigma^2}m^2}\right)v=0 \label{KGFinale}
\end{align}
We define the operators $h$ and $k$ such that the equation becomes $(\partial_t^2 - 2ik\partial_t + h)v=0$ and $h_0:= h+k^2$. 
Note that we can then reformulate \eqref{KGFinale} as a first order in time system: \[\partial_t \begin{pmatrix}v_0 \\ v_1 \end{pmatrix} = i\begin{pmatrix}0 & 1 \\ h & 2k\end{pmatrix} \begin{pmatrix} v_0 \\ v_1 \end{pmatrix}\]. 

\label{energySpace}
We recall that $h_0$ is a non negative injective selfadjoint operator on $L^2(\R_x\times\mathbb{S}^2)$ with domain $\left\{ u\in L^2(\R_x\times\mathbb{S}^2): h_0u\in L^2(\R_x\times\mathbb{S}^2) \right\}$ where the derivatives are taken in the distribution sense. We can now define the inhomogeneous energy space $\mathcal{E}:= Dom\left(h^{\frac{1}{2}}_0\right)\oplus L^2(\R_x\times\mathbb{S}^2)$ endowed with the norm $\left\|u\right\|^2_{\mathcal{E}}:= \left<(h_0+1) u_0, u_0\right> + \left\|u_1 - ku_0\right\|^2_{L^2}$. We also define the homogeneous energy space $\dot{\mathcal{E}}$ as the completion of $\mathcal{E}$ for the norm $\left\|u\right\|^2_{\dot{\mathcal{E}}}:= \left<h_0 u_0, u_0\right> + \left\|u_1 - ku_0\right\|^2_{L^2}$.
\begin{remark}
Let $u\in \left(C^{\infty}_0(\Sigma_0)\cap \mathcal{Y}^n\right)^2$, we denote by $f(t)$ the solution of the Cauchy problem for the Klein-Gordon equation with initial data $f(0) = u_0$ and $D_t f(0)= u_1$.
A tedious computation shows that $\frac{1}{2}\left\|u\right\|^2_{\dot{\mathcal{E}}}$ corresponds to the flux of the contraction between the stress energy tensor $T(f)$ and the vector field $X$ through $\Sigma_0$ as explained in section \ref{mainRes}
\end{remark}

According to section 3 of \cite{GGH}, the operator
\begin{equation}
H = \begin{pmatrix}0 & 1\\  h & 2k\end{pmatrix}
\end{equation} acting on $(C^{\infty}_0(\R \times \mathbb{S}^2))^2$ admits a closure (still called $H$) as an unbounded operator on $\mathcal{E}$ and a closure $\dot{H}$ as an unbounded operator on $\dot{\mathcal{E}}$. Moreover, $iH$ (resp. $i\dot{H}$) is the generator of a $C^0$-group on $\mathcal{E}$ (resp. on $\dot{\mathcal{E}}$). We call it $e^{itH}$ (resp. $e^{it\dot{H}}$). $e^{it\dot{H}}$ coincides with the continuous extension of $e^{itH}$ to $\dot{\mathcal{E}}$.
We denote by $H^n$ (resp. $\dot{H}^n$) the operator induced on $\mathcal{E}^n:= \mathcal{E}\cap (\mathcal{Y}^n)^2$ (resp. on $\dot{\mathcal{E}}^n$ the completion of $\mathcal{E}^n$ for the norm $\left\|.\right\|_{\dot{\mathcal{E}}}$).

We also define\footnote{The definition of $P$ in \cite{GGH} (section 12.2) is with a $\lambda^2$, but the right definition to get a smooth operator on the sphere is with a $\lambda$} the selfadjoint operator $P$ on $L^2(\mathbb{S}^2)$
\[ P = -\frac{\lambda}{\sin^2\theta}\partial_{\phi}^2 - \frac{1}{\sin\theta}\partial_{\theta} \sin \theta \Delta_{\theta}\partial_{\theta} \]
with domain 
\[Dom(P) = \left\{ u \in L^2(\mathbb{S}^2) : Pu \in L^2(\mathbb{S}^2) \right\} \]
where $Pu$ in this definition is understood as a distribution. With this definition, $P$ is a smooth elliptic differential operator selfadjoint on $L^2(\mathbb{S}^2)$. As a consequence, its spectrum is purely punctual (eigenvalues of finite multiplicity) and its eigenfunctions are smooth.
We denote by $(\lambda_q)_{q\in \N}$ the eigenvalues of $P$ and by $(Z_q)$ the associated eigenspaces. In the Hilbert sense, $\oplus_{q\in\N}Z_q = L^2(\mathbb{S}^2)$. Note that $P$ can also be viewed naturally as an operator acting on $L^2(\R\times \mathbb{S}^2)$ with eigenspaces $\mathcal{Z}_q:= L^2(\R)\otimes Z_q$ verifying $\oplus_{q\in \mathbb{N}}\mathcal{Z}_q = L^2(\R\times \mathbb{S}^2)$ in the hilbert sense. 

To simplify notations, we define $l = \frac{an}{a^2+r^2}$. Note that we have $l_+ = l(r_+)$ (resp. $l_-=l(r_-)$). We finally recall \footnote{Pay attention to the fact that the correct definition for $k_{\pm\infty}$ is $-l_{\pm}$ and not $l_{\pm}$ as in the article \cite{GGH}} the definition of the separable comparison operators defined in subsection 12.2 of \cite{GGH}: \newline
$h_{\pm \infty} = -l_{\pm}^2 - \partial_x^2 + \frac{\Delta_r}{\lambda^2(r^2 + a^2)^2}P + \Delta_r m^2$, \hspace{1.0cm} $k_{\pm \infty}:=-l_{\pm}$,\hspace{1.0cm} $h_{0,\pm \infty}:= h_{\pm \infty} + k_{\pm \infty}^2$\newline
 where $l_{\pm}:= \frac{an}{r_{\pm}^2+a^2}$. 

Following what we did previously for the Klein-Gordon energy spaces and operators (replacing $h$ by $h_{\pm \infty}$ and $k$ by $k_{\pm \infty}$), we define the spaces $\mathcal{E}_{\pm \infty}$, $\dot{\mathcal{E}}_{\pm \infty}$, $\mathcal{E}^n_{\pm \infty}$ and $\dot{\mathcal{E}}^n_{\pm \infty}$ as well as the closed operators $H_{\pm \infty}$ and $\dot{H}_{\pm \infty}$, $H^n_{\pm \infty}$
 and $\dot{H}^n_{\pm \infty}$. Note that $\dot{H}_{\pm\infty}$ is selfadjoint (see \cite{GGH} for more details).

We also introduce $i_{+}$ and $i_{-}$ which are, as in \cite{GGH}, smooth functions with $i_{+} = 1$, $i_{-} = 0$ in the neighborhood of $+\infty$, $i_{+} = 0$, $i_{-} = 1$ in the neighborhood of $-\infty$ and $i_{+}^2 + i_{-}^2 = 1$. 


\subsection{Definition of the comparison dynamics}
The goal of this work is to compare the natural Klein-Gordon dynamics on the De Sitter-Kerr space-time with the transport dynamics along principal null geodesics. The first difficulty that appears is the non commutation of $w_+$ and $w_-$ (unlike in the Schwarzschild case). Due to this fact, the solutions of the equation:
\begin{equation}
\frac{1}{2}(v_+v_- + v_-v_+)u = 0
\end{equation}
cannot be written as the sum of a part transported along incoming null geodesics and a part transported along outgoing null geodesics. We can write the equation that corresponds to these transports but it introduces new terms which are difficult to analyse. Instead, we will consider two different dynamics (one associated with the incoming transport and one associated with the outgoing transport).

\begin{definition}[dynamics associated with the outgoing transport]
\label{defOut}
We define the operator $\tilde{w}_- = -\partial_x - il + 2il_+$ (designed to commute with $w_+$ but with the same limit as $w_-$ in $r_+$). We consider the following equation:
\begin{equation}
(\partial_t + w_+)(\partial_t + \tilde{w}_-)u = 0
\end{equation}
which can be rewritten as a first order system
\begin{equation}
\label{outgoingDynamics}
\partial_t\begin{pmatrix}u_0 \\ u_1\end{pmatrix} = i\begin{pmatrix} 0 & 1 \\ h_{T,+} & 2k_{T,+}\end{pmatrix}\begin{pmatrix} u_0\\ u_1 \end{pmatrix}
\end{equation}
where $h_{T,+} = -(\partial_x + i(l-l_+))^2 - l_+^2$ and $k_{T,+} = -l_+$.
We call $H_{T,+}$ the matrix that appears in \eqref{outgoingDynamics} and $h_{0,T,+} = h_{T,+} + k_{T,+}^2$
\end{definition}

\begin{definition}[dynamics associated with the incoming transport]
We define the operator $\tilde{w}_+ = \partial_x - il + 2il_-$ (designed to commute with $w_-$ but with the same limit as $w_+$ in $r_-$). We consider the following equation:
\begin{equation}
(\partial_t + w_-)(\partial_t + \tilde{w}_+)u = 0
\end{equation}
which can be rewritten as a first order system
\begin{equation}
\label{incomingDynamics}
\partial_t\begin{pmatrix}u_0 \\ u_1\end{pmatrix} = i\begin{pmatrix} 0 & 1 \\ h_{T,-} & 2k_{T,-}\end{pmatrix}\begin{pmatrix} u_0\\ u_1 \end{pmatrix}
\end{equation}
where $h_{T,-} = -(-\partial_x + i(l-l_-))^2 - l_-^2$ and $k_{T,-} = -l_-$.
We call $H_{T,-}$ the matrix that appears in \eqref{incomingDynamics} and $h_{0,T,-} = h_{T,-} + k_{T,-}^2$
\end{definition}

\begin{remark}
The operators $h_{0,T,\pm}$, unbounded on $L^2(\R \times \mathbb{S}^2)$ with domain $\left\{\right. u\in L^2(\R \times \mathbb{S}^2), h_{0,T,+} u \in L^2(\R \times \mathbb{S}^2)\left.\right\}$ (where the derivative is understood in the distribution sense), are selfadjoint and non negative by classical arguments. Moreover lemma \ref{density} shows that the operators acting on $C^{\infty}_0$ are essentially selfadjoint.
\end{remark}
\begin{remark}
Because the analysis of both dynamics are very similar, we will focus here on the outgoing dynamics.
\end{remark}
\begin{remark}
Note that $H_{T,+}$ and $H_{T,-}$ depend on $n$ through $l$ and $l_{\pm}$. However we are not interested in uniformity with respect to $n$ in this paper. Therefore we keep this dependency implicit to alleviate the notations.
\end{remark}

\subsection{Spaces associated with the dynamics and energy operator}
Note that even though the angular momentum $n$ is fixed, we define $n$-dependent operators acting on the whole $L^2(\R\times\mathbb{S}^2)$ to simplify the arguments. However, we are interested in their action on the space of data with angular momentum equal to $n$.

%

We now define the energy spaces associated with the dynamics (on the model of what we did for the Klein-Gordon dynamics). We denote by $\mathcal{H}^s$ the (inhomogeneous) Sobolev spaces associated with $h_{0,T,+}$ (selfadjoint operator by the previous proposition), that is to say, the space $Dom(h_{0,T,+}^{s/2})$ endowed with the norm $\left\|(1+h_{0,T,+})^{s/2}\cdot\right\|_{L^2}$ if $s\geq 0$ and $(\mathcal{H}^s)^{*}$ if $s<0$. 
We also define $\mathcal{E}_{T,+}$ as the space $\mathcal{H}^{\frac{1}{2}} \times \mathcal{Y}^n$ endowed with the norm 
\begin{equation}
\left\| \begin{pmatrix} u_0 \\ u_1 \end{pmatrix}\right\|^2_{\mathcal{E}_{T,+}} = \left\| u_1 + l_+ u_0\right\|^2_{L^2} + \left<(h_{0,T,+}+1)u_0, u_0 \right>
\end{equation}

\begin{lemma}
\label{injectivity}
$\mathcal{H}^{1} = H^1_x:=\left\{ u\in L^2(\R\times\mathbb{S}^2), \partial_x u \in L^2(\R \times \mathbb{S}^2)\right\}$ and $h_{0,T,+}^{\frac{1}{2}}$ is injective on $\mathcal{H}^{1}$.
\end{lemma}

\begin{proof}
Let $u\in Dom\left(h_{0,T,+}^{\frac{1}{2}}\right)$.Let $u_k \in C^{\infty}_0$ be such that $\lim\limits_{k\to \infty} u_k = u$ for the graph norm (see lemma \ref{density} for the existence). In particular, $\left(h_{0,T,+}^{\frac{1}{2}}u_k\right)$ is a Cauchy sequence in $L^2$ and the equalities $\left\|h_{0,T,+}^{\frac{1}{2}}(u_k-u_{k'})\right\|^2_{L^2} = \left<h_{0,T,+}(u_k-u_{k'}),(u_k-u_{k'})\right> = \left\|(D_x + (l-l_+))(u_k-u_{k'})\right\|^2_{L^2}$ show that $\left((D_x + (l-l_+))u_k\right)$ is also Cauchy in $L^2$. The limit in $L^2$ is exactly $(D_x + (l-l_+))u$ (understood in the distribution sense) by uniqueness of the limit for the distribution topology. This proves $\mathcal{H}^1 \subset H^1_x$ and $\left\|h_{0,T,+}^{\frac{1}{2}}u\right\|^2_{L^2} = \left\|(D_x + (l-l_+))u_k\right\|^2_{L^2}$. The other inclusion can be proved in the same way.
Now the injectivity: let $u\in \mathcal{E}_{T,+}$ be such that $h_{0,T,+}^{\frac{1}{2}}u = 0$. Then $D_x u +(l-l_+)u=0$. If we write $v = e^{i\int_{0}^{x}(l-l_+)ds}u$, then $v \in H^1_x$ verifies $D_x v = e^{i\int_{0}^{x}(l-l_+)ds}((l-l_+)u+ D_x u) = 0$. So $v= 0$ and then $u=0$.
\end{proof}

We finally define $\dot{\mathcal{E}}_{T,+}$ as the completion of $\mathcal{E}_{T,+}$ for the norm 
\begin{equation}
\left\| \begin{pmatrix} u_0 \\ u_1 \end{pmatrix}\right\|^2_{\mathcal{\dot{E}}_{T,+}} = \left\| u_1 + l_+ u_0\right\|^2_{L^2} + \left<h_{0,T,+}u_0, u_0 \right>
\end{equation}
Note that the inclusion $\mathcal{E}_{T,+}\hookrightarrow \dot{\mathcal{E}}_{T,+}$ is continuous and dense. 
\begin{prop}
\label{isometryHom}
The linear map
\begin{equation}
L:
\begin{cases}
\mathcal{E}_{T,+} \rightarrow L^2(\R\times\mathbb{S}^2)\oplus L^2(\R\times\mathbb{S}^2) \\
\begin{pmatrix} u_0 \\u_1 \end{pmatrix} \mapsto \begin{pmatrix} h_{0,T,+}^{\frac{1}{2}}u_0 \\ u_1 +l_+ u_0\end{pmatrix} 
\end{cases}
\end{equation}
extends continuously as a linear bijective isometry between $\dot{\mathcal{E}}_{T,+}$ and $L^2(\R\times\mathbb{S}^2)\oplus L^2(\R\times\mathbb{S}^2)$.
\end{prop}
\begin{proof}
The extension and isometry property follows from the definition of $\dot{\mathcal{E}}_{T,+}$. To show the surjectivity of an isometry, it is enough to prove that its range is dense. Let $(v_0, v_1)\in L^2(\R\times\mathbb{S}^2)\oplus L^2(\R\times\mathbb{S}^2)$, because $h^{\frac{1}{2}}_{0,T,+}\geq 0$ and injective, $v_0^{n}:=1_{[\frac{1}{n}, +\infty)}\left(h_{0,T,+}^{\frac{1}{2}}\right)v_0$ converge towards $v_0$ in $L^2(\R\times\mathbb{S}^2)$. If we define the bounded function $f:=\frac{1}{x}1_{[\frac{1}{n}, +\infty)}(x)$ we can take $u_0^n:= f\left(h_{0,T,+}^{\frac{1}{2}}\right)v_0^n$ and $u_1^n := v_1 - l_+ u_0^n$ and $L(u^n) = (v_0^n, v_1)$ converges towards $v$ in $L^2(\R\times\mathbb{S}^2)\oplus L^2(\R\times\mathbb{S}^2)$.
\end{proof}

\begin{remark}
\label{remIsometryHom}
In this proof, if we assume that $v\in \mathcal{H}^1\oplus\mathcal{H}^1$, the constructed sequence converges towards $v$ for the natural norm on $\mathcal{H}^1 \oplus \mathcal{H}^1$.
\end{remark}

Finally we also define some spaces of data with angular momentum $n$ which behave well with respect to the operator $P$ (finite sum of eigenfunctions).
\begin{definition}
\begin{align*}
\mathcal{E}^{q,n}_{T,+} &= \mathcal{E}^{n}_{T,+}\cap \mathcal{Z}_q \oplus \mathcal{Z}_q \\ 
\mathcal{E}^{fin,n}_{T,+} &= \left\{ u\in \mathcal{E}^n_{T,+}: \exists Q\in \N / u\in \bigoplus_{q\leq Q}\mathcal{E}^{q,n}_{T,+}
\right\}
\end{align*}
And we get the corresponding spaces for $H_{\pm \infty}$ by replacing the indices $T,\pm$ by $\pm \infty$ in the definition.
\end{definition}

\begin{prop}
The operator (defined in definition \ref{defOut}) $H_{T,+}: C^{\infty}_0\rightarrow C^{\infty}_0$ (unbounded on $\dot{\mathcal{E}}_{T,+}$) has a selfadjoint extension with domain $\left\{u \in \dot{\mathcal{E}^n}: h^{\frac{1}{2}}_{0,T,+}u_0 \in \mathcal{H}^1, u_1 + l_+u_0\in \mathcal{H}^1\right\}$ given by \[\dot{H}_{T,+}\begin{pmatrix}u_0 \\ u_1\end{pmatrix} = 
\begin{pmatrix}u_1 \\
h_{0,T,+}u_0 - l^2_+u_0 - 2l_+u_1\end{pmatrix}\]

Moreover, $\mathcal{H}^2\oplus \mathcal{H}^1$ is dense in $Dom(H_{T,+})$ for the graph norm.
\end{prop}
\begin{proof}
We use the isometry $L$ defined in \ref{isometryHom} to reformulate the problem. We have $LH_{T,+}L^{-1} = \begin{pmatrix}-l_+ & h^{\frac{1}{2}}_{0,T,+} \\ h^{\frac{1}{2}}_{0,T,+} & -l_+ \end{pmatrix}$ which has a selfadjoint extension with domain $\mathcal{H}^1\oplus\mathcal{H}^1$ by classical arguments. So $H_{T,+}$ has a selfadjoint extension with domain $L^{-1}(\mathcal{H}^1\oplus \mathcal{H}^1) = \left\{u \in \dot{\mathcal{E}^n}: h^{\frac{1}{2}}_{0,T,+}u_0 \in \mathcal{H}^1, u_1 + l_+u_0\in \mathcal{H}^1\right\}$. For the density property, we use the remark \ref{remIsometryHom}.
\end{proof}

\begin{remark}
\label{Hinhom}
We can similarly show that the operator $H_{T,+}+\begin{pmatrix}0 & 0 \\ Id & 0\end{pmatrix}$ with domain $\mathcal{H}^2\oplus\mathcal{H}^1$ is selfadjoint on $\mathcal{E}_{T,+}$. Then as a consequence of the Hille-Yosida theorem, if we add the bounded operator $\begin{pmatrix}0 & 0 \\-Id & 0\end{pmatrix}$, we get $i$ times the generator of a $C^0$-group.  We denote it by $e^{itH_{T,+}}$. By classical $C^0$-group theory, there exist $A,B>0$ such that for all $u\in \mathcal{E}_{T,+}$ and for all $t\in \R_+$
\[
\left\| e^{itH_{T,+}}u\right\|_{\mathcal{E}_{T,+}}\leq Ae^{tB}\left\|u\right\|_{\mathcal{E}_{T,+}}
\]
Note that $\dot{H}_{T,+}$ is an extension of $H_{T,+}$. As a consequence, the two $C^0$-group $e^{itH_{T,+}}$ and $e^{it\dot{H}_{T,+}}$ coincide on $\mathcal{E}_{T,+}$. Indeed, for $u\in Dom(H_{T,+})$, the derivative of $e^{-it\dot{H}_{T,+}}e^{itH_{T,+}}u$ exists for the topology of $\dot{\mathcal{E}}_{T,+}$ (weaker of the two topologies) and is equal to zero. With the value at $t=0$ we get $e^{itH_{T,+}}u= e^{it\dot{H}_{T,+}}u$. For a general $u\in \mathcal{E}_{T,+}$, we take $(u_n)$ a  sequence of elements of $Dom(H_{T,+})$ converging to $u$ in $\mathcal{E}_{T,+}$. Then $e^{itH_{T,+}}u_n$ converges to $e^{itH_{T,+}}u$ in $\mathcal{E}_{T,+}$ but also to $e^{it\dot{H}_{T,+}}u$ in $\dot{\mathcal{E}}_{T,+}$. We conclude by uniqueness of the limit for the topology of $\dot{\mathcal{E}}_{T,+}$.
\end{remark}

\begin{remark}
We also introduce the space of data with angular momentum $n$ as $\mathcal{H}^1_n = \mathcal{H}^1\cap \mathcal{Y}^n$, $\mathcal{E}^n_{T,+} = \mathcal{H}^1_n\oplus \mathcal{Y}^n$ and $\dot{\mathcal{E}}^n_{T,+}$ the closure of $\mathcal{E}^n_{T,+}$ in $\dot{\mathcal{E}}_{T,+}$.
\end{remark}

We will often work with smooth compactly supported data. Therefore, we need density lemmas to recover more general results (see lemma \ref{density} and \ref{density2} of the appendix).

\subsection{Kirchoff formula}
In this section, we find an explicit expression $u(t)=\begin{pmatrix} u_0(t), u_1(t) \end{pmatrix}$ for the outgoing dynamics for smooth compactly supported initial data $\begin{pmatrix} u_0 \\u_1 \end{pmatrix}$.

\begin{definition}
We define $e^{-tw_+}$ and $e^{-t\tilde{w}_-}$ as the propagators of the dynamics generated on $L^2(\R\times \mathbb{S}^2)$ by the self-adjoint operators $-\frac{w_+}{i}$ (with domain $\mathcal{D}_{w_+}:=\left\{ u\in L^2: \partial_x u + ilu \in L^2 \right\}$) and $-\frac{\tilde{w}_-}{i}$ (with domain $\mathcal{D}_{\tilde{w}_-}:= \left\{ u \in L^2: -\partial_x u -ilu+2l_+u \in L^2 \right\}$).
\end{definition}
\begin{remark}
We have $C^{\infty}_0 \subset \bigcap_{k\in \N} Dom(w_+^k)$ and $C^{\infty}_0 \subset \bigcap_{k \in \N} Dom(\tilde{w}_-^k)$
\end{remark}

\begin{prop}
\label{transport}
We have the following expressions, for $f\in L^2(\mathbb{R}\times \mathbb{S}^2)$:
\begin{equation}
e^{-tw_+}f (x)= e^{i\int_{x}^{x-t}l(s)\dd s}f(x-t)
\end{equation}
\begin{equation}
e^{-t\tilde{w}_-}f(x) = e^{i\int_{x}^{x+t} l(s) - 2l_+\dd s}f(x+t)
\end{equation}
Analogous formulas hold for $\tilde{w}_+$ and $w_-$.
In particular, $e^{-tw_+}$ and $e^{-t\tilde{w}_-}$ send $C^{\infty}_0(\R \times \mathbb{S}^2)$ into itself and we have a propagation of the support. That is to say, if $supp(f) \subset (-R,R)$, then $supp(e^{-tw_+}f)\subset (-R+t, R+t)$ and $supp(e^{-\tilde{w}_-}f)\subset (-R-t, R-t)$.
\end{prop}

\begin{proof}
First, because both sides are continuous with respect to the $L^2$ topology, it is enough to prove the equality for $f\in C^{\infty}_0(\R \times \mathbb{S}^2)$.
To check that, we compute
\begin{align*}
\frac{\dd}{\dd t} e^{i\int_{x}^{x-t}l(s)ds}f(x-t) &= (-il(x-t)f(x-t) - (\partial_x f)(x-t)) e^{i\int_{x}^{x-t}l(s)ds}
\end{align*}
and
\begin{align*}
(-\partial_x - il)e^{i\int_{x}^{x-t}l(s)ds}f(x-t) =& -(il(x-t)f(x-t)-il(x)f(x-t) \\
&+ (\partial_x f)(x-t) + il(x))e^{i\int_{x}^{x-t}l(s)ds} \\
=& (-il(x-t)f(x-t) -(\partial_x f)(x-t))e^{i\int_{x}^{x-t}l(s)ds} \\
=&\frac{d}{dt} e^{i\int_{x}^{x-t}l(s)ds}f(x-t)
\end{align*}
We use the lemma \ref{continuity} to show that $t\mapsto e^{i\int_{\bullet}^{\bullet-t}l(s)ds}f(\bullet-t)\in C^1(\R, L^2)$ and that the punctual expression computed previously gives the derivative.
We can then use the lemma \ref{uniqueness} to get the conclusion.
Similarly, we prove the expression for $e^{-t\tilde{w}_-}$.
\end{proof}

Now we write explicitly the dynamics for a class of smooth compactly supported initial data. 
\begin{prop}
\label{kirchoff}
Let $u = \begin{pmatrix} u_0 \\u_1 \end{pmatrix} \in \left(C^{\infty}_0(\R \times \mathbb{S}^2)\right)^2$, such that the following condition holds:
\begin{equation}
\label{conditionInt}
\int_{-\infty}^{+\infty} e^{-i\int_{s}^{0}(l-l_+)}(u_1 + l_+ u_0)(s)\dd s = 0
\end{equation}
 Then the following equality holds:
\begin{equation}
e^{it\dot{H}_{T,+}}u = \frac{1}{2}\begin{pmatrix} e^{-tw_+}(u_0 + \tilde{u}_1) + e^{-t\tilde{w}_-}(u_0 - \tilde{u}_1)\\-\frac{w_+}{i}e^{-tw_+}(u_0 + \tilde{u}_1) -\frac{\tilde{w}_-}{i}e^{-t\tilde{w}_-}(u_0 - \tilde{u}_1)\end{pmatrix}
\end{equation}
where $\tilde{u}_1 (x)= -i\int_{-\infty}^{x} e^{-i\int_{s}^{x}(l-l_+)}(u_1 + l_+ u_0)(s)\dd s$
\end{prop}

\begin{remark}
The condition \eqref{conditionInt} ensures that $\tilde{u}_1$ has compact support. Indeed, since $u_1$ and $u_0$ are compactly supported, for $x$ large enough we have:
\[
\tilde{u}_1(x) = -i e^{-i\int_{0}^x(l-l_+)}\int_{-\infty}^{+\infty}e^{-i\int_s^0(l-l_+)}(u_1+l_+u_0)(s)\dd s
\]
\end{remark}

\begin{proof}
The idea once again is to apply the proposition \ref{uniqueness}. In what follows, we check the hypothesis of the lemma.
We call $f(t,x)$ the right hand side of the equality. For all $t \in \R$ we have $f(t,.) \in C^{\infty}_0(\R \times \mathbb{S}^2)$. So in particular for all $t \in \R$ we have $f(t,.) \in Dom(\dot{H}_{T,+})$.
We first compute $\partial_t f(t,x)$.
\begin{align*}
\partial_t f (t,x)&= \frac{1}{2}\begin{pmatrix}-w_+ e^{-tw_+}(u_0 + \tilde{u}_1)(x) - \tilde{w}_-e^{-t\tilde{w}_-}(u_0 - \tilde{u}_1)(x) \\
\frac{w_+^2}{i}e^{-tw_+}(u_0 + \tilde{u}_1)(x) + \frac{\tilde{w}_-^2}{i}e^{-t\tilde{w}_-}(u_0 - \tilde{u}_1)(x)
\end{pmatrix} \\
&= \begin{pmatrix}0 & i \\ i(w_+\tilde{w}_-) & -(w_+ + \tilde{w}_-) \end{pmatrix}f (t,.)(x)\\
&= i\dot{H}_{T,+}f (t,.)(x)
\end{align*}
It is a point-wise computation, we still need to prove that $f \in C^1(\R, \dot{\mathcal{E}}_{T,+})$ and that $\frac{\dd}{\dd t} f(t,.) = i\dot{H}_{T,+}f(t,.)$ for the topology on $\dot{\mathcal{E}}_{T,+}$. 
We write $f = \begin{pmatrix}f_0 \\f_1 \end{pmatrix}$
If we apply the lemma \ref{continuity} to $h_{0,T,+} f_0 \in L^2(\R\times \mathbb{S}^2)$ and to $f_1 + l_+ f_0\in L^2(\R\times\mathbb{S}^2)$. We get
 $h_{0,T,+} f_0 $ and $f_1 + l_+ f_0$ are in $C^1(\R, L^2)$ and that the derivative for the $L^2$ topology is given according to the previous point-wise computation. Because $f \mapsto (h_{0,T,+}f_0, f_1+l_+f_0)$ is an isometry from  $\dot{\mathcal{E}}$ to $L^2\times L^2$, we deduce that $f \in C^1(\R, \dot{\mathcal{E}}^n_{T,+})$ and that $\frac{\dd}{\dd t} f(t,.) = i\dot{H}_{T,+}f(t,.)$.
The last property to be checked is the initial value. For $t = 0$:
\begin{align*}
f(0, x) &= \begin{pmatrix} u_0 \\ -(\frac{w_+}{2i}+\frac{\tilde{w}_-}{2i})u_0 - (\frac{w_+}{2i}-\frac{\tilde{w}_-}{2i})\tilde{u}_1 \end{pmatrix} \\
&= \begin{pmatrix} u_0 \\ u_1 \end{pmatrix}
\end{align*}
%

So we can apply the proposition \ref{uniqueness} and conclude the proof.
\end{proof}


\section{Existence of the direct wave operators}
\subsection{Preliminary results: density, propagation and boundedness}
\begin{definition}
We define the following spaces:
\begin{itemize}
\item
\begin{align*}
\mathcal{E}^{n,L}_{T,+} =& \text{\huge\{} \begin{pmatrix} u_0 \\u_1 \end{pmatrix}\in\mathcal{E}^n_{T,+} : (u_1 + l_+u_0) \in L^1(\R, L^2(\mathbb{S}^2))\\
&\text{ and } \int_{-\infty}^{+\infty} e^{-i\int_{s}^{0}(l-l_+)}(u_1 + l_+ u_0)(s)\dd s = 0 \text{ a. e. in $\omega \in \mathbb{S}^2$} \text{\huge\} }
\end{align*}
\item
$\mathcal{E}^{fin,n,L}_{T,+} = \mathcal{E}^{n,L}_{T,+} \cap \mathcal{E}^{fin,n}_{T,+}$
\item
$\mathcal{D}^{fin}_{T,\pm}:= (C^{\infty}_0(\R_x \times \mathbb{S}^2))^2 \cap \mathcal{E}^{fin,n,L}_{T,\pm}$
\end{itemize}
\end{definition}

Note that The space $\mathcal{D}^{fin}_{T,\pm}:= (C^{\infty}_0(\R_x \times \mathbb{S}^2))^2 \cap \mathcal{E}^{fin,n,L}_{T,\pm}$ is dense in $\mathcal{E}^n_{T, \pm}$ (and thus in $\dot{\mathcal{E}}^n_{T, \pm}$ thanks to the continuous and dense inclusion). See lemma \ref{density3} for the details.

\begin{lemma}
\label{propagation}
For u$\in \mathcal{D}^{fin}_{T,+}$. We assume that $\text{supp} (u_0) \cup \text{supp} (u_1) \subset (-R, R)\times \mathbb{S}^2 $. Then there exists time dependent families $(u^l(t))$ and $(u^r(t))$ such that for all $t$, $u^{l}(t) \in \left(C^{\infty}_0(\R \times \mathbb{S}^2)\right)^2\cap \dot{\mathcal{E}}^{fin,n}_{T,+}$ and $u^{r}(t) \in \left(C^{\infty}_0(\R\times \mathbb{S}^2)\right)^2\cap \dot{\mathcal{E}}^{fin,n}_{T,+}$ and 
\[
\forall t\in \R, e^{it\dot{H}_{T,+}}u = u^{l}(t) + u^{r}(t)
\]
with the properties:
\begin{itemize}
\item
$supp \left(u^{l}(t)\right)\subset (-\infty, R-t)$ and $supp \left(u^{r}(t)\right)\subset (R+t, +\infty)$
\item
\[
\forall t \in \R, \left\|u^{l}_0(t)\right\|_{L^2} = \left\|u^{l}_0(0)\right\|_{L^2}
\]
and
\[
\forall t \in \R, \left\|u^{r}_0(t)\right\|_{L^2} = \left\|u^{r}_0(0)\right\|_{L^2}
\]
\end{itemize}
\end{lemma}

\begin{proof}
We use the explicit expression of $e^{it\dot{H}_{T,+}}u$ given by proposition \ref{kirchoff}. We remark that the condition \eqref{conditionInt} implies that $supp (\tilde{u}_1) \subset (-R, R)$.
We define \[u^{l}(t) = \frac{1}{2}\begin{pmatrix}e^{-\tilde{w}_-t}(u_0 - \tilde{u}_1) \\\frac{-\tilde{w}_-}{i}e^{-\tilde{w}_-t}(u_0 - \tilde{u}_1)\end{pmatrix}\] and \[ u^{r}(t) = \frac{1}{2}\begin{pmatrix}e^{-w_+t}(u_0 + \tilde{u}_1) \\\frac{-w_+}{i}e^{-w_+t}(u_0 + \tilde{u}_1)\end{pmatrix}\]
With this definition, the first condition of the lemma is a direct consequence of the proposition \ref{transport}. For the second condition, we use the fact that $\frac{w_+}{i}$ and $\frac{\tilde{w}_-}{i}$ are selfadjoint on $L^2$. 
\end{proof}

Until the end of this section, we use the notation $\begin{pmatrix}u_0(t) \\u_1(t) \end{pmatrix}:= e^{it\dot{H}_{T,+}}u$.
The following boundedness lemma will be useful to differentiate $e^{-it\dot{H}^n_{+\infty}}i_+ e^{it\dot{H}_{T,+}}u$, for $u\in D^{fin}_{T,+}$.
\begin{lemma}
\label{boundedness2}
Let $Q \in \N$. For all $u\in \bigoplus_{q\leq Q}\mathcal{E}^{q,n}_{T,+}\cap (C^{\infty}_0 (\R \times \mathbb{S}^2))^2 $, 
\[
\left\|i_+ e^{it\dot{H}_{T,+}}u\right\|_{\dot{\mathcal{E}}^n_{+\infty}} \leq C(Q)e^{Bt}\left\|u \right\|_{\mathcal{E}^n_{T,+}}
\]
where $B>0$ and $C(Q)>0$
\end{lemma}

\begin{proof}
Let $u\in \bigoplus_{q\leq Q}\mathcal{E}^{q,n}_{T,+} $.
The fact that $P$ commutes with $i_+e^{itH_{T,+}}$ gives that $i_+ e^{it\dot{H}^n_{T,+}}u \in \bigoplus_{q\leq Q}\mathcal{E}^{q,n}_{T,+}$, in particular for any function $f\in C^{\infty}(\R_x)$
\[
|\left<f^2 P u_0(t) , u_0(t)\right>|\leq C_0(Q) \left\|f(x)u_0(t)\right\|^2_{L^2}
\]
for some constant $C_0(Q)>0$.

We use this in the following computation
\begin{align*}
\left\| i_+e^{it\dot{H}_{T,+}}u \right\|^2_{\dot{\mathcal{E}}^n_{+\infty}} &= \left\|i_+u_1(t) + i_+l_+ u_0(t)\right\|^2_{L^2} - \left< \partial_x^2 i_+u_0(t), i_+u_0(t)\right> \\
&+ \left<\frac{\Delta_r}{\lambda^2(r^2+a^2)^2}Pi_+ u_0(t) + \Delta_r m^2 i_+ u_0(t), u_0(t)\right> \\
&\leq \left\| u(t) \right\|^2_{\dot{\mathcal{E}}^n_{T,+}} + C_1(Q) \left\| u_0(t)\right\|^2_{L^2}\\
&\leq C_2(Q) \left\| e^{it\dot{H}^n_{T,+}}u\right\|_{\mathcal{E}^n_{T,+}}
\end{align*}
We then use the remark \ref{Hinhom} to conclude.
\end{proof}

\subsection{Existence of the direct wave operators}
\begin{theorem}
\label{existenceDirect}
There exists $a_0>0$ such that for all $|a| < a_0$ and for all $n\in \Z$ the following holds:
for all $\phi \in \mathcal{E}^{fin,n,L}_{T,\pm}\cap (C^{\infty}_0(\R \times \mathbb{S}^2))^2$, the following limit exists in $\dot{\mathcal{E}}$:
\[                                       
\lim\limits_{t\to \infty} e^{-it\dot{H}^{n}}i_{\pm}^2 e^{it\dot{H}_{T,\pm} }\phi
\]
Moreover, the limit operator extends continuously as an operator from $\dot{\mathcal{E}}^n_{T,+}$ to $\dot{\mathcal{E}}^n$.
\end{theorem}

\begin{proof}
Let $u\in D^{fin}_{T,+}$. 
Thanks to theorem 12.2 of \cite{GGH}, it is enough to prove the existence of the limit of 
\[
\lim\limits_{t\to \infty} e^{-it\dot{H}^n_{+\infty}}i_+ e^{it\dot{H}_{T,+}}u
\]
Thanks to proposition \ref{kirchoff}, we see that
$i_+e^{it\dot{H}^n_{T,+}}u \in (C^{\infty}_0(\R \times \mathbb{S}^2))^2$ . In particular for all $t\in\R$, $i_+e^{it\dot{H}_{T,+}}u \in \dot{\mathcal{E}}^n_{+\infty}$. 

Using this and the lemma \ref{boundedness2}, we can strongly differentiate the function $g(t) = e^{-it\dot{H}^n_{+\infty}}i_+ e^{it\dot{H}_{T,+}}u$.
We get 
\begin{align*}
\frac{\dd}{\dd t} g(t) &= e^{-it\dot{H}^n_{+\infty}} \left(i_+ iH_{T,+} - iH^n_{+\infty} i_+\right) e^{it\dot{H}_{T,+}}u
\end{align*}
We define
\begin{align*}
\delta &:= i_+ iH_{T,+} - iH^n_{+\infty} i_+ \\
& = \begin{pmatrix} 0& 0 \\ i\left(i_+h_{T,+} - h_{+\infty}i_+\right) & 0\end{pmatrix}
\end{align*}
We can compute further
\begin{align*}
\frac{\delta_{2,1}}{i} :=& i_+h_{T,+} - h_{+\infty}i_+\\
=& \partial_x^2(i_+) + 2\partial_x(i_+)\partial_x + i_+(l-l_+)^2 -2ii_+(l-l_+)\partial_x \\
&- i i_+ (\partial_x l) - i_+\frac{\Delta_r}{\lambda^2(r^2 + a^2)^2}P - i_+\Delta_r m^2
\end{align*}
All the terms in this expression decay at least like $Ce^{-\kappa_+ x}$ when $x\rightarrow +\infty$ and are zero in a neighborhood of $-\infty$ (thanks to the $i_+$ function). Indeed, we recall that
\begin{align*}
l-l_+ &= \frac{na}{r^2+a^2} - \frac{na}{r_+^2+a^2}\\
&= (r-r_+) \alpha(r)
\end{align*}
with $\alpha(r)$ bounded (and $r-r_+ \leq e^{-\kappa_+x}$) and
\begin{align*}
\partial_x l &= -\frac{2nr\Delta_r}{\lambda (a^2+r^2)^3}\\
&\leq (r-r_+) \beta(r)
\end{align*}
with $\beta(r)$ bounded.

We define $R>0$ such that $supp (u_0), supp (u_1) \subset (-R, R) $. We also denote by $Q$ an integer such that $u\in \oplus_{q\leq Q} \mathcal{E}^{q,n}_{T,+}$
We now use the lemma \ref{propagation} in the following computation:
\begin{align*}
\left\|\delta e^{it\dot{H}_{T,+}}u \right\|^2_{\dot{\mathcal{E}}^n} & = \left\|\delta_{2,1} (u^{in}_0(t)+ u^{out}_0(t)) \right\|^2_{L^2} \\
& \leq C(Q) e^{-2\kappa_+ (t-R)}2\left(\left\|u^{in}_0(0)\right\|^2_{L^2} + \left\|u^{out}_0(0)\right\|^2_{L^2}\right)
\end{align*}
We deduce that $\left\|\frac{\dd}{\dd t} f\right\|_{\dot{\mathcal{E}}^n}$ is integrable. So $\lim\limits_{t \to \infty}f(t)$ exists in $\dot{\mathcal{E}}^n$. 
To conclude the proof of the first part of the theorem, we use the theorem 12.2 of \cite{GGH} (and the fact that the family $e^{-it\dot{H}^n}i_+e^{it\dot{H}_{+\infty}}$ is uniformly bounded).

We now prove the extension property. Let $u\in \mathcal{D}^{fin}_{T,+}$. Recall that, thanks to the theorem 12.1 in \cite{GGH}, there exists $C_n>0$ independent from $u$ such that
\begin{align*}
\left\|e^{-it\dot{H}^n}i^2_+ e^{it\dot{H}_{T,+}}u\right\|_{\dot{\mathcal{E}}^n} &\leq C_n \left\| i^2_+ e^{it\dot{H}_{T,+}}u\right\|_{\dot{\mathcal{E}}^n}  \\
& \leq C'_n \left\| i_+ e^{it\dot{H}_{T,+}}u\right\|_{\dot{\mathcal{E}}^n_{+\infty}}
\end{align*}
The second equality following from the lemma 5.4 (and lemma 10.2 to see that $\dot{\mathcal{E}}_+ = \dot{\mathcal{E}}_{+\infty}$) of the article \cite{GGH}.
\begin{align*}
\left\| i_+ e^{it\dot{H}_{T,+}}u\right\|^2_{\dot{\mathcal{E}}^n_{+\infty}}  = & \left\|i_+(\partial_x + i(l-l_+)) u_0(t) - i(l-l_+)i_+u_0(t) + \partial_x(i_+)u_0(t) \right\|^2_{L^2} \\
&+ \left\|i_+(u_1(t) + l_+ u_0(t))\right\|^2_{L^2}+ \left< \left(\frac{\Delta_r}{\lambda^2(r^2+a^2)^2}P + \Delta_r m^2\right)i_+ u_0(t),i_+ u_0(t) \right>\\
\leq & 2\left\|u\right\|^2_{\dot{\mathcal{E}}^n_{T,+}} + 2\left\|-i(l-l_+)i_+u_0(t) + \partial_x(i_+)u_0(t)\right\|^2_{L^2} \\
&+ \left< \left(\frac{\Delta_r}{\lambda^2(r^2+a^2)^2}P + \Delta_r m^2\right)i_+ u_0(t),i_+ u_0(t) \right> \\
\leq & 2\left\|u\right\|^2_{\dot{\mathcal{E}}^n_{T,+}} + C(Q)\left<f(x) u_0(t), u_0(t)\right>
\end{align*}
where $f$ is a smooth function $O(e^{-\kappa_+ x})$ when $x\rightarrow +\infty$ and zero in a neighborhood of $-\infty$ (thanks to the $i_+$ cutoff). We use the lemma \ref{propagation} and we get (calling $R$ a real such that $supp (u_0), supp(u_1) \subset (-R, R)$):
\begin{align*}
\left\| i_+ e^{it\dot{H}_{T,+}}u\right\|^2_{\dot{\mathcal{E}}^n_{+\infty}}& \leq \left\|u\right\|^2_{\dot{\mathcal{E}}^n_{T,+}} + 2C(Q)e^{-\kappa_+ (t-R)}\left(\left\|u^{in}_0(0)\right\|^2_{L^2} + \left\|u^{out}_0(0)\right\|^2_{L^2}\right)
\end{align*}
Letting $t\rightarrow +\infty$, we have the following bound for the limit operator $W_{T,+}$:
\[
\left\|W_{T,+}u\right\|_{\dot{\mathcal{E}}^n}\leq C'_n \left\|u\right\|_{\dot{\mathcal{E}}^n_{T,+}}
\]
With the density lemma \ref{density3}, we get that $W_{T,+}$ extends uniquely to $\dot{\mathcal{E}}^n_{T,+}$.
\end{proof}

\section{Existence of the inverse wave operators}
\subsection{Preliminary results: boundedness, compatibility of domains}

Now we prove the existence of inverse wave operators.
We do in advance some computation that will be used later
\begin{lemma}
\label{computation}
We assume that $n\neq 0$ or $m>0$.
For $u\in Dom(\dot{H}^n_{+\infty})\cap \mathcal{E}^{n,q}_{+\infty}$ such that $i_+ u \in Dom(\dot{H}^n_{T,+})\cap \mathcal{E}^{n,q}_{+\infty}$ (so in particular for $u \in (C^{\infty}_0(\R \times \mathbb{S}^2))^2 \cap \mathcal{E}^{n,q}_{+\infty}$ ) we define
\begin{align*}
\delta' u :&= i(i_+\dot{H}^n_{+\infty} - \dot{H}_{T,+}i_+) u \\
& = \begin{pmatrix} 0 \\ \delta'_{2,1}u_0  \end{pmatrix}
\end{align*}
where 
\begin{align}
\frac{\delta'_{2,1}}{i} u_0 = &  i_+ \frac{\Delta_r}{\lambda^2(a^2+ r^2)^2}\lambda_q u_0 + i_+ \Delta_r m^2 u_0+ (\partial_x^2 i_+) u_0 + 2 (\partial_x i_+)(\partial_x u_0) \notag\\
& + 2(\partial_x i_+) i(l-l_+) u_0+ 2i_+i(l-l_+)(\partial_x u_0) + i_+ i (\partial_x l) u_0 - i_+ (l-l_+)^2u_0
\label{formuleDelta21}
\end{align}

Then the following inequalities hold:
\begin{enumerate}
\item \label{avecPoids}
\[
\left\|\delta'_{2,1} u_0\right\|^2_{L^2}\leq C_q \left\|h_{0,+\infty}^{\frac{1}{2}}\sqrt{q(r)}u_0 \right\|^2_{L^2} \\
\]
\item \label{sansPoids}
\[
\left\|\delta'_{2,1} u_0\right\|^2_{L^2}\leq C_q \left\|h_{0, +\infty}^{\frac{1}{2}} u_0 \right\|^2_{L^2}
\]
\end{enumerate}
where $q(r) = \sqrt{(r_+ -r)(r-r_-)}$

\end{lemma}

\begin{proof}
We remark that 
\begin{align*}
\delta'_{2,1} u_0 &= f(x) \partial_x u_0 + g_q(x) u_0 \\
\end{align*}
where $f, g_q \in C^{\infty}(\R)$ are $O(q(r)^2)$. So the second inequality is proved since 
\[ \left\|h_{0, +\infty}^{\frac{1}{2}} u_0 \right\|^2_{L^2} = \left\|\partial_x u_0 \right\|^2_{L^2} + \left\|\sqrt{\Delta_r\left(\frac{\lambda_q}{\lambda^2(a^2+r^2)^2} + m^2\right)} u_0 \right\|^2_{L^2} \]
with $\lambda_q \geq 0$ (strictly if $n\neq 0$ and we assume $m^2>0$ if $n=0$) and $\sqrt{\Delta_r}\geq C q(r)$.

To prove the first inequality:
\begin{align*}
\delta_{2,1} u_0 &= f(x)q(r)^{-\frac{1}{2}} \partial_x (q(r)^{\frac{1}{2}} u_0) + q(r)^{-\frac{1}{2}}(g_q(x) - f(x)q(r)^{-\frac{1}{2}}(\partial_x q(r)^{\frac{1}{2}})) (q(r)^{\frac{1}{2}} u_0) \\
& = f^1(x) \partial_x (q(r)^{\frac{1}{2}} u_0) + g^1_q(x)(q(r)^{\frac{1}{2}} u_0)
\end{align*}
where $f^1,g_q^1 \in C^{\infty}(\R)$ are in $O(q(r)^{\frac{3}{2}})$. So we get the first inequality (and even slightly better).
\end{proof}

Then we need a boundedness lemma:
\begin{lemma}
\label{boundedness3}
We still assume that $n\neq 0$ or $m>0$.
The multiplication operator \[i_{\pm}:(C^{\infty}_0 (\R \times \mathbb{S}^2))^2 \cap \left(\mathcal{Y}^n\right)^2 \rightarrow (C^{\infty}_0 (\R \times \mathbb{S}^2))^2 \cap \left(\mathcal{Y}^n\right)^2  \] extends to a bounded operator \[i_{\pm}:\dot{\mathcal{E}}^n_{\pm\infty} \rightarrow \dot{\mathcal{E}}^n_{T,\pm}\] 
It follows that for all $u\in\dot{\mathcal{E}}^n$, 
\[
\left\| i_{\pm}^2 e^{it\dot{H}^n}u \right\|_{\dot{\mathcal{E}}_{T,\pm} }\leq C_n \left\| u \right\|_{\dot{\mathcal{E}}^n }
\]
and finally
 $e^{-it\dot{H}^n_{T,\pm}}i_{\pm}^2 e^{it\dot{H}^n} \in \mathcal{B}(\dot{\mathcal{E}}^n, \dot{\mathcal{E}}^n_{T,\pm})$ is uniformly bounded.
\end{lemma}
\begin{proof}
As usual, we prove only the $+$ case.
Let $u\in (C^{\infty}_0 (\R \times \mathbb{S}^2))^2 \cap \left(\mathcal{Y}^n\right)^2$.
\begin{align*}
\left\| i_+ u \right\|^2_{\mathcal{E}^n_{T,+} } &= \left\| i_+u_1 + i_+l_+u_0\right\|^2_{L^2}+\left\|(\partial_x+i(l-l_+))i_+u_0\right\|^2_{L^2} \\
&\leq \left\| i_+u_1 + i_+l_+u_0\right\|^2_{L^2} + 2(\left\|\partial_x i_+u_0\right\|^2_{L^2} + \left\|i(l-l_+)i_+u_0\right\|^2_{L^2}) \\
&\leq 2\left\| u \right\|^2_{\dot{\mathcal{E}}_{+\infty}} + \left<f(x)u_0, u_0\right>
\end{align*}
where $f$ is a smooth function in $O(e^{-\kappa_+ x})$ when $x\rightarrow +\infty$ and zero in a neighborhood of $-\infty$ (thanks to the $i_+$ cutoff).
So $\left<f(x)u_0,u_0\right>$ can be bounded by the term $\left<\Delta_r m^2 u_0, u_0\right>$ (or if $m=0$ and $n\neq 0$ by the term $\left<\frac{\Delta_r}{\lambda^2(a^2+r^2)^2}P u_0, u_0 \right> $ ) in the norm $\left\| u\right\|_{\dot{\mathcal{E}}^n_{+\infty}}$. In each case:
\[|\left<f(x)u_0, u_0\right>|\leq C\left\| u\right\|^2_{\dot{\mathcal{E}}^n_{+\infty}}
\]
and we deduce:
\[
\left\| i_+ u \right\|^2_{\dot{\mathcal{E}}^n_{T,+} } \leq (C+2)\left\|  u \right\|^2_{\dot{\mathcal{E}}^n_{+\infty}}
\]
Using the density lemma \ref{densityFin}, we deduce that $i_+$ has a unique bounded extension from $\dot{\mathcal{E}}^n_{+\infty}$ to $\dot{\mathcal{E}}^n_{T,+}$.

Finally, we use lemma 5.4 (and 10.2) of the article \cite{GGH} (boundedness of $i_+: \dot{\mathcal{E}}^n\rightarrow \dot{\mathcal{E}}^n_{+\infty}$) to conclude 
\[
\left\| i_+^2 u \right\|_{\mathcal{E}^n_{T,\mp} } \leq C'\left\|  u \right\|_{\dot{\mathcal{E}}^n}
\]
\end{proof}

In order to differentiate $e^{-it\dot{H}^n_{T,\pm}}i_{\pm}^2 e^{it\dot{H}^n}$, we have to prove the following lemma
\begin{lemma}
\label{domain}
We assume that $n\neq 0$ or $m>0$. \newline
$i_+(Dom(\dot{H}^n_{+\infty})\cap \mathcal{E}^{q, n}_{+\infty})\subset Dom(\dot{H}_{T,+})$ and therefore
$i_+(Dom(\dot{H}^n_{+\infty})\cap \mathcal{E}^{fin, n}_{+\infty})\subset Dom(\dot{H}_{T,+})$
\end{lemma}
\begin{proof}
For $u\in (C^{\infty}_0 (\R \times \mathbb{S}^2))^2 \cap  \mathcal{E}^{q, n}_{+\infty}$,
we compute
\begin{align*}
\left\|i_+ u \right\|^2_{\dot{\mathcal{E}}^n_{T,+}} + \left\|\dot{H}_{T,+}i_+u\right\|^2_{\dot{\mathcal{E}}^n_{T,+}} \leq & \left\|i_+ u \right\|^2_{\dot{\mathcal{E}}^n_{T,+}} + 2\left\|i_+\dot{H}^n_{+\infty}u\right\|^2_{\dot{\mathcal{E}}^n_{T,+}} + 2\left\|\delta' u\right\|^2_{\dot{\mathcal{E}}^n_{T,+}}
\end{align*}
%
%

The first two terms can be bounded using lemma \ref{boundedness3}. The lemma \ref{computation} (inequality number \ref{sansPoids}) gives $\left\|\delta' u\right\|^2_{\dot{\mathcal{E}}^n_{T,+}}\leq \left\|u\right\|^2_{\dot{\mathcal{E}}^n}$.
Therefore we get
\[
\left\|i_+u\right\|_{graph(\dot{H}_{T,+})} \leq C(q) \left\|u\right\|_{graph(\dot{H}^n_{+\infty})}
\]
where $\|.\|_{graph(\dot{H}_{T,+})}$ (resp. $\|.\|_{graph(\dot{H}^n_{+\infty})}$) denotes the graph norm associated with the operator $\dot{H}_{T,+}$ (resp. $\dot{H}^n_{+\infty}$).
We can use the second part of the density lemma \ref{densityFin} to conclude the proof.
\end{proof}

\subsection{Existence of the inverse wave operators}
\begin{theorem}
There exists $a_0>0$ such that, for all $|a|<a_0$ and for all $n\neq 0$,
the following strong limits exist in $\mathcal{B}(\dot{\mathcal{E}}^n, \dot{\mathcal{E}}^n_{T,\pm})$:
\[
\Omega_{T, \pm}=s-\lim\limits_{t\to \infty} e^{-it\dot{H}^n_{T,\pm}}i_+^2 e^{it\dot{H}^n}
\]
\end{theorem}

\begin{remark}
The limitation $n\neq 0$ comes from lemma 13.2 of \cite{GGH} that we will need. However, the case $n=0$ with $m^2>0$ is covered by the same method.
\end{remark}

\begin{proof}
Thanks to theorem 12.2 of \cite{GGH} (and the uniform boundedness proved in lemma \ref{boundedness3}), it is enough to prove the existence of the limit of 
\[
s-\lim\limits_{t\to \infty} e^{-it\dot{H}^n_{T,+}}i_+ e^{it\dot{H}^n_{+\infty}}
\]

We define the notation $F(t):=  e^{-it\dot{H}^n_{T,+}}i_+ e^{it\dot{H}^n_{+\infty}}$
We begin by proving the convergence of $F(t)\chi(\dot{H}^n_{+\infty})u$ for $u \in (C^{\infty}_0 (\R\times \mathbb{S}^2))^2\cap \mathcal{E}^{fin,n}_{+\infty}$ and $\chi \in C^{\infty}_0(\R)$.
By linearity, it is enough to prove it for $u \in (C^{\infty}_0 (\R\times \mathbb{S}^2))^2\cap \mathcal{E}^{q,n}_{+\infty}$. 
Moreover, $\chi(\dot{H}^n_{+\infty})$ and $e^{it\dot{H}^n_{+\infty}}$ preserve $Dom(\dot{H}^n_{+\infty})$ and $\mathcal{E}^{q,n}_{+\infty}$. So we deduce $e^{it\dot{H}^n_{+\infty}}\chi(\dot{H}^n_{+\infty})u \in Dom(\dot{H}^n_{+\infty}) \cap \mathcal{E}^{q,n}_{+\infty}$. So thanks to the lemma \ref{domain}, $i_+ e^{it\dot{H}^n_{+\infty}} \in Dom(e^{-it\dot{H}^n_{T,+}})$ for all $t\in \R$. As a consequence we can compute (with the notation $\delta'$ from lemma \ref{computation})
\[
\frac{\dd }{\dd t}F(t)\chi(\dot{H}^n_{+\infty})u = e^{-it\dot{H}^n_{T,+}}\delta' e^{it\dot{H}^n_{+\infty}}\chi(\dot{H}^n_{+\infty})u
\]
Because, $\dot{H}_{T,+}$ is selfadjoint, we can write
\begin{align*}
\left\|\frac{\dd }{\dd t}F(t)\chi(\dot{H}^n_{+\infty})u\right\|_{\dot{\mathcal{E}}^n_{T,+}} &= \left\|\delta' e^{it\dot{H}^n_{+\infty}}\chi(\dot{H}^n_{+\infty})u\right\|_{\dot{\mathcal{E}}^n_{T,+}} \\
&\leq C_q\left\|\sqrt{q(r)} e^{it\dot{H}^n_{+\infty}}\chi(\dot{H}^n_{+\infty})u\right\|_{\dot{\mathcal{E}}^n_{+\infty}}\\
&\leq \left\|\sqrt{q(r)} e^{it\dot{H}^n_{+\infty}}\chi(\dot{H}^n_{+\infty})\sqrt{q(r)}\right\|_{\mathcal{B}(\dot{\mathcal{E}}_{+\infty})}\left\|q(r)^{-\frac{1}{2}}u\right\|_{\dot{\mathcal{E}}^n_{+\infty}}
 \end{align*}
where the first inequality is obtained by lemma \ref{computation} inequality number \ref{avecPoids}.
Because, $u\in (C^{\infty}_0 (\R\times \mathbb{S}^2))^2 $ we have $\left\|q(r)^{-\frac{1}{2}}u\right\|_{\dot{\mathcal{E}}^n_{+\infty}}<+\infty$.
Then we can use the analogous of lemma 6.7 of \cite{GGH} for $H_{+\infty}$ 
to show that $\left\|\frac{\dd}{\dd t}F(t)\chi(\dot{H}^n_{+\infty})u\right\|_{\dot{\mathcal{E}}^n_{T,+}}$ is integrable and therefore, the limit exists.

Using the lemma \ref{boundedness3} (uniform boundedness) and a density argument, we recover the limit for $u\in \dot{\mathcal{E}}^n$. 

\end{proof}

\section{Construction and properties of the global wave operators}
\subsection{Construction of the global wave operators}
In the following sections, we fix $n\in \Z$ and assume that $n\neq 0$ or $m>0$.
\subsection{Left and right spaces}
\begin{definition}
We define the following subspaces of $\dot{\mathcal{E}}_{T,+}$ and $\dot{\mathcal{E}}_{T,-}$.
\[ \mathcal{E}^{r}_{T,+} = \left\{ \begin{pmatrix}u_0 \\ u_1 \end{pmatrix} \in \dot{\mathcal{E}}^n_{T,+} : \frac{w_+}{i}u_0 + u_1 = 0 \right\} \]
\[
\mathcal{E}^{l}_{T,+} = \left\{ \begin{pmatrix}u_0 \\ u_1 \end{pmatrix} \in \dot{\mathcal{E}}^n_{T,+} : \frac{\tilde{w}_-}{i}u_0 + u_1 = 0 \right\}
\]
\[ \mathcal{E}^{r}_{T,-} = \left\{ \begin{pmatrix}u_0 \\ u_1 \end{pmatrix} \in \dot{\mathcal{E}}^n_{T,+} : \frac{\tilde{w}_+}{i}u_0 + u_1 = 0 \right\} \]
\[ \mathcal{E}^{l}_{T,-} = \left\{ \begin{pmatrix}u_0 \\ u_1 \end{pmatrix} \in \dot{\mathcal{E}}^n_{T,+} : \frac{w_-}{i}u_0 + u_1 = 0 \right\}
\]
\end{definition}

\begin{remark}
\label{defExt}
Functions like $u\mapsto\frac{w_+}{i}u_0 + u_1$ are well defined from $\mathcal{E}_{T,+}$ to $L^2(\R\times\mathbb{S}^2)$  and the equality
\begin{equation}
\frac{w_+}{i}u_0 + u_1 := \frac{1}{i}(\partial_x + i(l-l_+))u_0 + (u_1 + l_+u_0)
\end{equation}
shows that they admit a unique continuous extension from $\dot{\mathcal{E}}_{T,+}$ to $L^2(\R\times\mathbb{S}^2)$ (which is given by the right-hand side).
\end{remark}

\begin{prop}
The subspaces $\mathcal{E}^{l}_{T,\pm}$ and $\mathcal{E}^{r}_{T,\pm}$ are closed.
\end{prop}

\begin{proof}
We treat the case of $\mathcal{E}^{r}_{T, +}$
The application
\[
\Psi:
\begin{cases}
\dot{\mathcal{E}}_{T,+} \rightarrow L^2 \\
u \mapsto \frac{1}{i}(\partial_x + i(l-l_+))u_0 + (u_1 + l_+u_0)
\end{cases}
\]
is continuous and $\mathcal{E}^{r}_{T, +} = ker \Psi$
\end{proof}

\begin{prop}
\label{decomposition}
We have the following decompositions:
\[
\dot{\mathcal{E}}^n_{T,\pm} = \mathcal{E}^{l}_{T,\pm} \oplus \mathcal{E}^{r}_{T,\pm}
\]
\end{prop}

\begin{lemma}
\label{pythagore}
For all $u^{l}\in \mathcal{E}^{l}_{T,\pm}$ and $u^{r}\in\mathcal{E}^{r}_{T,\pm}$,
\[
\|u^{l}+u^{r}\|^2_{\dot{\mathcal{E}}_{T,\pm}} =\|u^{l}\|^2_{\dot{\mathcal{E}}_{T,\pm}} + \|u^{r}\|^2_{\dot{\mathcal{E}}_{T,\pm}}
\]
\end{lemma}

\begin{remark}
The subspaces are in fact orthogonal to one another with respect to the natural scalar product on $\dot{\mathcal{E}}_{T,\pm}$. 
\end{remark}

\begin{proof}[Proof of the lemma]
We treat the $+$ case. 
We use the following equalities (direct from the definition)
\begin{align*}
(\partial_x + i(l-l_+))u^{l}_0 &= i(u^{l}_1 + l_+u^{l}_0)\\
(\partial_x + i(l-l_+))u^{r}_0 &= -i(u^{r}_1 + l_+u^{r}_0)
\end{align*}
in the following computation
\begin{align*}
\|u^{l}+u^{r}\|^2_{\dot{\mathcal{E}}^n_{T,\pm}} &= \|(\partial_x + i(l-l_+))(u^{l}_0 + u^{r}_0)\|^2 + \|u^{l}_1 + l_+ u^{l}_0 + u^{r}_1+ l_+u^{r}_0\|^2\\
&= \|u^{l}_1 + l_+ u^{l}_0 - (u^{r}_1 + l_+u^{r}_0)\|^2 + \|u^{l}_1 + l_+ u^{l}_0 + u^{r}_1+ l_+u^{r}_0\|^2 \\
&= 2\|u^{l}_1 + l_+ u^{l}_0\|^2 + 2\|u^{r}_1 + l_+ u^{r}_0\|^2 \\
&= \|u^{l}\|^2_{\dot{\mathcal{E}}^n_{T,\pm}} + \|u^{r}\|^2_{\dot{\mathcal{E}}^n_{T,\pm}}
\end{align*}
\end{proof}

\begin{proof}[Proof of the proposition]
As usual, we treat only the $+$ case (the other being similar). 
The trivial intersection property follows from the lemma.
Now we prove that $\dot{\mathcal{E}}^n_{T,+} = \mathcal{E}^{l}_{T,+} + \mathcal{E}^{r}_{T,+}$.
From the lemma (and the fact that $\mathcal{E}^{l}_{T,+}$ and $\mathcal{E}^{r}_{T,+}$ are closed in a complete space, hence complete), we get that $\mathcal{E}^{l}_{T,+} + \mathcal{E}^{r}_{T,+}$ is complete, hence closed.
We remark that the space $\mathcal{D}^{fin}_{T,+}$ defined in lemma \ref{density3} is included in the sum $\mathcal{E}^{l}_{T,+} + \mathcal{E}^{r}_{T,+}$. Indeed, for $u\in \mathcal{D}^{fin}_{T,+}$, we can write:
\[ \begin{pmatrix} u_0 \\ u_1 \end{pmatrix} = \begin{pmatrix} u_0 + \tilde{u}_1 \\ -\frac{w_+}{i} (u_0 + \tilde{u}_1)\end{pmatrix} + \begin{pmatrix}u_0 - \tilde{u}_1 \\ -\frac{\tilde{w}_-}{i}(u_0 - \tilde{u}_1)\end{pmatrix}
\]
where 
\[
\tilde{u}_1 = -i\int_{-\infty}^{x}e^{-i\int_s^x (l-l_+)}(u_1+l_+u_0)(s)\dd s
\]
is smooth compactly supported. So we have $\mathcal{D}^{fin}_{T,+}\subset \mathcal{E}^l_{T,+} + \mathcal{E}^r_{T,+}$. We conclude the proof, using the density of $\mathcal{D}^{fin}_{T,+}$ (lemma \ref{density3}).
\end{proof}
\begin{remark}
\label{decompoReg}
We remark for further use that we even have $\mathcal{D}^{fin}_{T,+}\subset\mathcal{E}^l_{T,+}\cap (C^{\infty}_0)^2 + \mathcal{E}^r_{T,+}\cap (C^{\infty}_0)^2$.
\end{remark}

The comparison operator $H_{T,+}$ (resp. $H_{T,-}$) is a good approximation of $\dot{H}$ only near $r=r_+$ (resp. near $r= r_-$). That is why we have put a cutoff in the definition of the wave operators. For this reason we have to understand how to glue them together to construct global wave operators. The first step is to define the profile space, which will be the image of the global inverse wave operator.
\begin{definition}
We define the profile space as follows:
\[ \mathcal{P} = \mathcal{E}^{l}_{T,-}\oplus \mathcal{E}^{r}_{T,+} \] 
\end{definition}

\begin{definition}
We define the global direct wave operator
\[
W: \begin{cases}
\mathcal{P} \rightarrow \dot{\mathcal{E}}^n \\
(u^l, u^r) \mapsto W_{T,-}(u^l)+ W_{T,+}(u^r)
\end{cases}
\]
\end{definition}

\begin{definition}
We define the global inverse wave operator
\[
\Omega: \begin{cases}
\dot{\mathcal{E}}^n \rightarrow \dot{\mathcal{E}}^n_{T,-}\oplus\dot{\mathcal{E}}^n_{T,+}\\
u \mapsto (\Omega_{T,-}u, \Omega_{T,+}u)
\end{cases}
\]
\end{definition}

\begin{remark}
We remark that $\mathcal{P}\subset \dot{\mathcal{E}}^n_{T,-}\oplus\dot{\mathcal{E}}^n_{T,+}$. We will see later that $\mathcal{P}$ is exactly the image of $\Omega$, therefore it is the correct profile space. 
\end{remark}

\subsection{Inversion property}
\label{properties}
The goal of this section is to prove the following theorem
\begin{theorem}
\label{inversion}
$Ran(\Omega) \subset \mathcal{P}$ and 
\begin{align*}
\Omega W &= Id_{\mathcal{P}} \\
W \Omega &= Id_{\dot{\mathcal{E}}^n}
\end{align*}
\end{theorem}

The proof will be split in three main independent parts:
\begin{itemize}
\item The proof of $Ran(\Omega) \subset \mathcal{P}$.
\item The proof of $\Omega W = Id_{\mathcal{P}}$.
\item The proof of the injectivity of $\Omega$
\end{itemize}
From these three steps we deduce the theorem.

\subsubsection{First part: $Ran(\Omega) \subset \mathcal{P}$}
The claim can be split into two distinct inclusions:
\begin{prop}
\label{RanOmega}
\begin{align*}
Ran(\Omega_{T,+}) &\subset \mathcal{E}^r_{T,+} \\
Ran(\Omega_{T,-}) &\subset \mathcal{E}^l_{T,-}
\end{align*}
\end{prop}
As usual we only prove the $+$ case, the other being very similar.
In the following we use the already introduced notation
\[
\Psi:
\begin{cases}
\dot{\mathcal{E}}_{T,+} \rightarrow L^2 \\
u \mapsto \frac{1}{i}w_+ u_0 + u_1
\end{cases}
\]
and we recall that $\Psi$ is continuous (see remark \ref{defExt})

We also recall that $ker(\Psi) = \mathcal{E}^{r}_{T,+}$. For this reason it is interesting to study the $L^2$-norm of $\Psi$ along the dynamics.
\begin{lemma}
\label{invariance}
For all $t$ in $\mathbb{R}$ and all $u$ in $\dot{\mathcal{E}}_{T,+}$,
\[
\|\Psi(e^{it\dot{H}_{T,+}}u)\|_{L^2} = \|\Psi(u)\|_{L^2}
\]
In other words, the $L^2$ norm of $\Psi$ is preserved along the comparison dynamics $e^{itH_{T,+}}$.
\end{lemma}
\begin{remark}
In particular $e^{it\dot{H}_{T,+}}$ preserves $\mathcal{E}^{r}_{T,+}$.
\end{remark}
\begin{proof}
By continuity of $\Psi$ and $\Psi\circ e^{it\dot{H}_{T,+}}$, it is enough to prove the result on the dense subset $D^{fin}_{T,+}$ defined in lemma \ref{density3} because we know (see remark \ref{decompoReg}) that every element $u\in D^{fin}_{T,+}$ can be written as
\[
u = u^l+u^r
\]
with $u^l \in (C^{\infty}_0)^2 \cap \mathcal{E}^l_{T,+}$ and $u^r\in (C^{\infty}_0)^2 \cap \mathcal{E}^l_{T,+}$. 
 Moreover thanks to the Kirchoff formula, we have the explicit expression
\[
e^{it\dot{H}_{T,+}}u = \begin{pmatrix} e^{-t\tilde{w}_-}u^l_0 + e^{-tw_+}u^r_0 \\ -\frac{\tilde{w}_-}{i}e^{-t\tilde{w}_-}u^l_0 - \frac{w_+}{i}e^{-tw_+}u^r_0 \end{pmatrix}
\]
Then we compute
\begin{align*}
\|\Psi(e^{it\dot{H}_{T,+}}u)\|_{L^2} &= \| \frac{1}{i}(w_+-\tilde{w}_-)e^{-it\tilde{w}_-}u^l_0 \|_{L^2} \\
& = \| 2(\partial_x + i(l-l_+))e^{-it\tilde{w}_-}u^l_0 \|_{L^2}
\end{align*}
But $2(\partial_x + i(l-l_+))$ commutes with $\tilde{w}_-$ and so with $e^{-t\tilde{w}_-}$. Finally, using that$\frac{1}{i}\tilde{w}_-$ is selfadjoint on $L^2$ (and therefore $e^{-t\tilde{w}_-}$ is unitary), we get
\begin{align*}
\|\Psi(e^{it\dot{H}_{T,+}}u)\|_{L^2} &= \|2(\partial_x + i(l-l_+))u^l_0 \|_{L^2} \\
&= \| \frac{w_+}{i} u^l_0 + u^l_1\|_{L^2}
\end{align*}
\end{proof}

As in \cite{GGH}, we introduce the weight $w:= \frac{1}{\sqrt{(r-r_-)(r_+-r)}}$.
\begin{definition}
We define the set $\mathcal{A} = \left\{ \chi (\dot{H}_{+\infty})w^{-\epsilon} u, u \in \dot{\mathcal{E}}^{fin,n}_{+\infty}, \chi \in C^{\infty}_0 \right\}$, which correspond to data with good propagation property with respect to the dynamics associated to $\dot{H}_{+\infty}$.
\end{definition}
$\mathcal{A}$ is dense in $\dot{\mathcal{E}}_{+\infty}$ (see lemma \ref{densityA} for details).

\begin{remark}
Because $\dot{H}_{+\infty}$ commutes with $P$, we have $\mathcal{A}\subset \dot{\mathcal{E}}^{fin,n}_{+\infty} \cap Dom(H_{+\infty})$
\end{remark}

Finally we state a technical lemma which is useful for the computations.
\begin{lemma}
\label{tech}
The application $\Psi_{|Dom(\dot{H}_{T,+})}$ is a continuous application from $Dom(\dot{H}_{T,+})$ to $Dom(\tilde{w}_-)$. So for $u\in Dom(\dot{H}_{T,+})$, $\tilde{w}_-\Psi(u)$ is well defined (in the distribution sense) and belongs to $L^2$. Moreover, we have:
\[
\tilde{w}_- \Psi(u) = -i((\dot{H}_{T,+}u)_1 + l_+ (\dot{H}_{T,+}u)_0) - (\partial_x + i(l-l_+))(\dot{H}_{T,+}u)_0
\]
\end{lemma}
\begin{proof}
let $u \in (C^{\infty}_0)^2$. 
\begin{align}
\MoveEqLeft[4] -i((\dot{H}_{T,+}u)_1 + l_+ (\dot{H}_{T,+}u)_0) -  (\partial_x + i(l-l_+))(\dot{H}_{T,+}u)_0 &= -i(\tilde{w}_-w_+u_0 -2l_+u_1 + l_+u_1)- (\partial_x + i(l-l_+))u_1 \notag\\
&= \tilde{w}_-\frac{w_+}{i}u_0 + (il_+ - \partial_x -il+il_+)u_1 \notag\\
&= \tilde{w}_-\left(\frac{w_+}{i}u_0 + u_1\right) \notag\\
&=\tilde{w}_-\Psi(u) \label{calc}
\end{align}

If we take a general $u\in Dom(\dot{H}_{T,+})$, then by density we can find $(u_n)$ a sequence of functions in $(C^{\infty}_0)^2$ converging towards $u$ in the graph norm (see for example the much stronger lemma \ref{density2}). By continuity of $\Psi$,
\[
\Psi (u_n) \xrightarrow[n\to +\infty]{} \Psi(u)
\]
In $L^2$, therefore in the distribution sense. We deduce that
\[
\tilde{w}_- \Psi (u_n) \xrightarrow[n\to +\infty]{} \tilde{w}_-\Psi(u)
\]
in the distribution sense. Finally, the continuity of the left hand side of \eqref{calc} from $Dom(H_{+\infty})$ to $L^2$, we get that $\lim\limits_{n \to +\infty}\tilde{w}_- \Psi (u_n) = ((\dot{H}_{T,+}u)_1 + l_+ (\dot{H}_{T,+}u)_0) +  (\partial_x + i(l-l_+))(\dot{H}_{T,+}u)_0$ in $L^2$ (and therefore in the distribution topology). We conclude by uniqueness of the limit in the distribution topology.
\end{proof}

We can now prove the proposition \ref{RanOmega}.
\begin{proof}[Proof of proposition \ref{RanOmega}]
We begin by some reductions.
We recall (see the proof of existence of $\Omega_{T,+}$) that $\Omega_{T,+} = \Omega_{T,+}^{+\infty} \Omega_{+\infty}$ where
\begin{align*}
\Omega_{T,+}^{+\infty} &= s-\lim\limits_{t\to +\infty} e^{-it\dot{H}_{T,+}}i_+e^{it\dot{H}^n_{+\infty}} \\
\Omega_{+\infty} &= s-\lim\limits_{t \to +\infty} e^{-it\dot{H}^n_{+\infty}}i_+ e^{it\dot{H}^n}
\end{align*} 
So it is enough to prove that $Ran(\Omega_{T,+}^{+\infty}) \subset \mathcal{E}^r_{T,+}$.
Because $\mathcal{E}^r_{T,+}$ is closed in $\dot{\mathcal{E}}_{T,+}$ and $\Omega_{T,+}^{+\infty}$ is continuous, it is enough to prove that $\Omega_{T,+}^{+\infty}(\mathcal{A}) \subset \mathcal{E}^r_{T,+}$. The reduction to $\mathcal{A}$ enables to use the propagation estimate (as we do later).



\vspace{0.5cm}
Let $u$ be in $\mathcal{A}$. We denote by $Q$ an integer such that $u \in \oplus_{q\leq Q} \mathcal{\dot{E}}^{q,n}_{+\infty}$. Using the remark after lemma \ref{densityA}, we get $e^{it\dot{H}^n_{+\infty}}u \in Dom(\dot{H}_{+\infty})\cap \oplus_{q\leq Q}\dot{\mathcal{E}}^{q,n}_{+\infty}$. The lemma \ref{domain} gives us that $i_+e^{it\dot{H}^n_{+\infty}}u \in Dom(\dot{H}_{T,+})$.
We can apply the lemma \ref{tech} and get $\Psi(i_+e^{it\dot{H}_{+\infty}}u) \in  Dom(\tilde{w}_-)$. We also remark that $t\mapsto \Psi(i_+ e^{it\dot{H}_{+\infty}} u)$ is in $C^1(\mathbb{R}; L^2(\mathbb{R}\times \mathbb{S}^2))$ (using $u\in Dom(\dot{H}_{+\infty})$).
To alleviate the notations we define $u(t):= e^{it\dot{H}^n_{+\infty}}u$ and by $u^q(t)$ the orthogonal projection of $u(t)$ on $\dot{\mathcal{E}}^{q,n}_{+\infty}$.
We can compute (see the proof of \ref{tech} for the $\tilde{w}_-$ computation)
\begin{align*}
(\partial_t + \tilde{w}_-)\Psi(i_+u(t)) =& \Psi(i_+ i\dot{H}_{+\infty} u(t)) - i((\dot{H}_{T,+}i_+u(t))_1 + l_+ (\dot{H}_{T,+}i_+u(t))_0) \notag \\
&- (\partial_x + i(l-l_+))(\dot{H}_{T,+}i_+u(t))_0 
\end{align*}
\vspace{0.5cm}
We now show that for all $v \in Dom(\dot{H}_{+\infty})\cap \dot{\mathcal{E}}^{q,n}_{+\infty}$,
\begin{align}
\label{eqtech1}
\Psi(i_+ i\dot{H}_{+\infty} v) - i((\dot{H}_{T,+}i_+v)_1 + l_+ (\dot{H}_{T,+}i_+v)_0)- (\partial_x + i(l-l_+))(\dot{H}_{T,+}i_+v)_0 &= \delta_{2,1}'v_0  
\end{align}
We recall that $\delta_{2,1}':= i(i_+ h_{+\infty} - h_{T,+}i_+)$ was defined in lemma \ref{computation}.
We emphasize the fact that the condition $v \in Dom(\dot{H}_{+\infty})\cap \dot{\mathcal{E}}^{q,n}_{+\infty}$ implies that both sides of \eqref{eqtech1} belongs to $L^2$.  
Both sides are continuous with respect to the graph norm of $\dot{H}_{+\infty}$ (for the left-hand side we use the estimation in the proof of lemma \ref{domain}). Then by density (see lemma \ref{densityFin}), it is enough to prove it for $v\in (C^{\infty}_0)^2 \cap \dot{\mathcal{E}}^{q,n}_{+\infty}$. 
After this reduction, we can compute without paying attention to domains

\begin{align*}
\MoveEqLeft \Psi(i_+ iH_{+\infty} v) - i((H_{T,+}i_+v)_1
+ l_+ (H_{T,+}i_+v)_0)- (\partial_x + i(l-l_+))(H_{T,+}i_+v)_0 &= 
(\partial_x+il)(i_+ v_1) + ii_+h_{+\infty}v_0\\ &\phantom{{}={}} - 2ii_+l_+ v_1 
- i (\partial_x + il)(-\partial_x-il+2il_+) (i_+v_0)\\ 
&\phantom{{}={}}+ 2il_+ i_+ v_1 -i l_+i_+v_1 - (\partial_x + i(l-l_+))(i_+v_1) \\[0.3cm]
&= i \left(i_+h_{+\infty}v_0 - (\partial_x + il)(-\partial_x -il+2l_+)(i_+v_0)\right)\\[0.3cm]
&= \delta_{2,1}'v_0 
\end{align*}
where $\delta'_{2,1}$ is defined in lemma \ref{computation}. \\
Then we use the equality for each $u^q(t)$ and we get
\begin{align*}
f(t):&=(\partial_t + \tilde{w}_-)\Psi(i_+u(t)) \\
&=  \sum_{q\leq Q} \delta_{2,1}' u^q_0(t)
\end{align*}
Remark that $\delta_{2,1}'$ and $\sqrt{q(r)}$ commute with $P$ (and thus the image of an eigenfunction of $P$ is still an eigenfunction associated with the same eigenvalue).
We use inequality \ref{avecPoids} of lemma \ref{computation} and the $L^2$-orthogonality of eigenspaces of $P$ in the following computation:
\begin{align*}
\left\|f(t)\right\|^2_{L^2}&\leq \sum_{q\leq Q}C_q\left\|h_{0,+\infty}^{\frac{1}{2}}\sqrt{q(r)}u^q_0\right\|^2_{L^2} \\
&\leq \max_{q\leq Q}C_q \left\|\sqrt{q(r)}u\right\|^2_{\dot{\mathcal{E}}_{+\infty}}
\end{align*}
Because $u\in\mathcal{A}$, we can use the propagation estimate for $\dot{H}^n_{+\infty}$ (the equivalent of inequality (6.9) in proposition 6.7 of \cite{GGH}). We get
\[
\int_{0}^{+\infty}\|f(t)\|_{L^2}\dd t <+\infty
\]
The interpretation of this integrability is that $\Psi(i_+u(t))$ almost propagates towards $\left\{r=r_-\right\}$.
\vspace{0.5cm}
We now consider $g(t) := e^{t\tilde{w}_-}\Psi(i_+u(t))$. Let $h>0$.
\begin{align*}
\frac{g(t+h)-g(t)}{h} = \frac{(e^{(t+h)\tilde{w}_-}-e^{t\tilde{w}_-})}{h}\Psi(i_+u(t)) + e^{(t+h)\tilde{w}_-}\left(\frac{\Psi(i_+u(t+h))-\Psi(i_+u(t))}{h}\right)
\end{align*}
The first term has a limit in $L^2$ when $h\rightarrow 0$ because $\Psi(i_+u(t))\in Dom(\tilde{w}_-)$ (by lemma \ref{tech}). The second term also has a limit because $e^{t\tilde{w}_-}$ is uniformly bounded on $L^2$ (even unitary) and $t\mapsto \Psi(i_+u(t))$ is in $C^1(\mathbb{R}_t;L^2)$ (already seen in this proof).
We deduce that $g$ is differentiable and we can compute the derivative:
\[\partial_t g(t) = e^{t\tilde{w}_-}(\partial_t \Psi(i_+u(t)) + \tilde{w}_- \Psi(i_+u(t))) = e^{t\tilde{w}_-}f(t) \]
With the initial condition, we deduce 
\[
g(t) = \int_0^{t} e^{s\tilde{w}_-}f(s)\dd s + \Psi(i_+u(0))
\]
Finally we get
\begin{align*}
\Psi(i_+u(t)) &= e^{-t\tilde{w}_-}\left(\int_0^{t} e^{s\tilde{w}_-}f(s)\dd s + \Psi(i_+u(0))\right)\\
&= e^{-t\tilde{w}_-}\left(\int_0^{+\infty} e^{s\tilde{w}_-}f(s)\dd s + \Psi(i_+u(0))\right) + o_{L^2}(1)
\end{align*}
In other words, $\Psi(i_+u(t))$ propagates towards $\left\{ r= r_- \right\}$ up to an error converging to $0$ for large $t$. We denote by $p:= \left(\int_0^{+\infty} e^{s\tilde{w}_-}f(s)\dd s + \Psi(i_+u(0))\right) \in L^2$

\vspace{0.5cm}
Now we can use this propagation to prove that $\Psi(i_+u(t))$ tends to zero in $L^2$. We define $\widetilde{i}_+$ such that $\widetilde{i}_+i_+ = i_+$ and $\widetilde{i}_+ = 0$ in a neighborhood of $r_-$ and $1$ on a neighborhood of $r_+$.
\begin{align*}
\Psi(i_+u(t)) &= \Psi(\widetilde{i}_+i_+u(t)) \\
&= \widetilde{i}_+ \Psi(i_+u(t)) + \partial_x (\widetilde{i}_+) i_+ u_0(t)\\
&= \widetilde{i}_+ \Psi(i_+u(t))
\end{align*}
because $\widetilde{i}_+ = 1$ on the support of $i_+$.
To conclude, we only have to show in $L^2$:
\[ \lim\limits_{t\to +\infty}  \widetilde{i}_+ e^{-t\tilde{w}_-}p = 0 \]
We use the lemma \ref{transport}:
\begin{align*}
\|\widetilde{i}_+ e^{-t\widetilde{w}_-}p\|^2_{L^2} &= \int_{\mathbb{R}} \widetilde{i}^2_+(x)|p(x+t)|^2 \dd x\\
&= \int_{\mathbb{R}} \widetilde{i}^2_+(x-t)|p(x)|^2 \dd x
\end{align*}
We check that $\forall x \in \R, \lim\limits_{t\to +\infty} \widetilde{i}^2_+(x-t)|p(x)|^2 = 0$ and 
 $\forall x \in \R, \lim\limits_{t\to +\infty} \widetilde{i}^2_+(x-t)|p(x)|^2 \leq |p(x)|^2$ with $|p|^2 \in L^1(\mathbb{R})$. So by Lebesgue domination theorem, we conclude 
\[
\lim\limits_{t \to +\infty}\|\Psi(i_+u(t))\|_{L^2} = 0
\]
\vspace{0.5cm}
finally, we use lemma \ref{invariance} to write
\[
\lim\limits_{t\to +\infty}\| \Psi(e^{-itH_{T,+}}i_+u(t))\|_{L^2} = 0
\]
and by continuity of $\Psi$, 
\[
\Psi(\Omega_{T,+}^{+\infty}u)=0
\]
which gives $\Omega_{T,+}^{+\infty}u \in \mathcal{E}^{r}_{T,+}$.
\end{proof}

\subsubsection{Second part: $\Omega W = id_{\mathcal{P}}$}
We first split the equality into smaller pieces. Let $(u^l,u^r)\in \mathcal{P}$.
\[
\Omega W(u^l, u^r) = (\Omega_{T,-} W_{T,-} u^l + \Omega_{T,-}W_{T,+} u^r, \Omega_{T,+}W_{T,-} u^l + \Omega_{T,+}W_{T,+} u^r)
\]
Therefore it is enough to prove the following proposition
\begin{prop}
For all $u^r\in \mathcal{E}^{r}_{T,+}$ and for all $u^l\in \mathcal{E}^{l}_{T,-}$,
\begin{align*}
\Omega_{T,+}W_{T,+} (u^r) &= u^r \\
\Omega_{T,+}W_{T,-} (u^l) &= 0 \\
\Omega_{T,-}W_{T,+}(u^r) &=0 \\
\Omega_{T,-}W_{T,-}(u^l) &= u^l
\end{align*}
\end{prop}

We need a density lemma first
\begin{lemma}
\label{densityLR}
$(C^{\infty}_0)^2\cap \mathcal{E}^{r}_{T,+}$ is dense in $\mathcal{E}^{r}_{T,+}$.
$(C^{\infty}_0)^2 \cap \mathcal{E}^{l}_{T,-}$ is dense in $\mathcal{E}^{l}_{T,-}$.
\end{lemma}

\begin{proof}
It is a corollary of the stronger density result, lemma \ref{density3}. Indeed, in the proof of proposition \ref{decomposition} we see $ \pi_r(\mathcal{D}^{fin}_{T,+}) \subset (C^{\infty}_0)^2 \cap \mathcal{E}^{r}_{T,+}$ where $\pi_r$ is the projection on $\mathcal{E}^{r}_{T,+}$ parallel to $\mathcal{E}^{l}_{T,+}$. This projection is continuous by the lemma \ref{pythagore}. The second density can be done in the same way.
\end{proof}

\begin{proof}[Proof of the proposition]
We only prove the first two equalities (the other being very similar). By density, we can prove them for $u^l \in (C^{\infty}_0)^2 \cap \mathcal{E}^{l}_{T,-}$ and $u^r \in (C^{\infty}_0)^2 \cap \mathcal{E}^{r}_{T,+}$.
We know (by theorem \ref{existenceDirect}) that in this case we have:
\begin{align*}
W_{T,+} u^r &= \lim\limits_{t\to +\infty} e^{-it\dot{H}^n}i_+ e^{it\dot{H}_{T,+}} u^r \\
W_{T,-} u^l &= \lim\limits_{t\to +\infty} e^{-it\dot{H}^n}i_- e^{it\dot{H}_{T,-}} u^l
\end{align*}
The fact that $e^{-it\dot{H}_{T,+}}i_+e^{it\dot{H}^n}$ is uniformly bounded (with respect to $t$) and that ${\Omega_+ = s-\lim\limits_{t \to +\infty} e^{-it\dot{H}_{T,+}}i_+e^{it\dot{H}^n}}$ enables to write:
\begin{align*}
\Omega_{T,+}W_{T,+} u^r &= \lim\limits_{t\to +\infty} e^{-it\dot{H}_{T,+}}i_+e^{it\dot{H}^n}e^{-it\dot{H}^n}i_+ e^{it\dot{H}_{T,+}} u^r = e^{-it\dot{H}_{T,+}}i_+^2 e^{it\dot{H}_{T,+}}u^r\\
\Omega_{T,+}W_{T,-} u^l &= \lim\limits_{t\to +\infty} e^{-it\dot{H}_{T,+}}i_+e^{it\dot{H}^n}e^{-it\dot{H}^n}i_- e^{it\dot{H}_{T,-}} u^l = e^{-it\dot{H}_{T,+}}i_+i_- e^{itH_{T,-}}u^l
\end{align*}
We can apply the propositions \ref{kirchoff} (and the version for $\dot{H}_{T,-}$) and \ref{transport} to find:
\begin{align*}
e^{it\dot{H}_{T,+}}u^r(x) &= \begin{pmatrix}e^{i\int_{x}^{x-t}l(s)\dd s}u^r_0(x-t) \\ -\frac{w_+}{i}e^{i\int_{x}^{x-t}l(s)\dd s}u^r_0(x-t)\end{pmatrix} \\
e^{it\dot{H}_{T,-}}u^l(x) &= \begin{pmatrix}e^{-i\int_{x}^{x+t}l(s)\dd s}u^l_0(x+t) \\ -\frac{\tilde{w}_-}{i}e^{-i\int_{x}^{x+t}l(s)\dd s}u^l_0(x+t)\end{pmatrix}
\end{align*}
We deduce that for $t$ large enough,
\begin{align*}
i_+^2 e^{it\dot{H}_{T,+}}u^r &= e^{it\dot{H}_{T,+}}u^r \\
i_+i_- e^{it\dot{H}_{T,-}}u^l &= 0
\end{align*}
We conclude the proof by density.
\end{proof}

\subsubsection{Third part: Injectivity of $\Omega$}
In this part, we need a version of the propagation estimate for general $u$. We obtain it as a corollary of proposition 6.7 of \cite{GGH} by a density argument.

\begin{lemma}[Propagation result for general $u$]
\label{propEstGen}
For all $\epsilon >0$.
For all $u\in \dot{\mathcal{E}}_{\pm \infty}$,
\[
\lim\limits_{t\to +\infty} \|w^{-\epsilon}e^{it\dot{H}_{\pm\infty}} u\|_{\dot{\mathcal{E}}_{\pm \infty}} = 0
\]
We have the same result for the $\dot{H}$ dynamics.
\end{lemma}


\begin{lemma}[Propagation of $L^2$ norm of the first component]
\label{propFirstComp}
We still assume that $n\neq 0$ or $m>0$.
For all $u\in \dot{\mathcal{E}}_{\pm \infty}$ and for all $\epsilon>0$, we have:
\[\lim\limits_{t\to +\infty} \|w^{-\epsilon}(e^{it\dot{H}_{\pm \infty}}u)_0 \|_{L^2} = 0
\]
\end{lemma}
\begin{remark}
The quantity $\|w^{-\epsilon}u_0 \|_{L^2}$ is well defined for $u\in (C^{\infty}_0)^2$. We extend the function by continuity for general $u\in\dot{\mathcal{E}}_{\pm \infty}$ thanks to the Hardy inequality (see lemma \ref{HardyIneq}) which gives the estimate:
\begin{align*}
\|w^{-\epsilon}u_0 \|_{L^2} &\leq C(\|\partial_x u_0 \|_{L^2} + \|u_0\|_{L^2((-1, 1))}) \\
&\leq C\|u\|_{\dot{\mathcal{E}}_{\pm \infty}}
\end{align*}
\end{remark}

\begin{proof}
The idea of the proof is to use the Hardy inequality with $w^{-\frac{\epsilon}{2}}$ and then to use the general estimate result shown in the previous lemma. To simplify the notation we write $u(t) := e^{it\dot{H}_{\pm \infty}}u$.
By Hardy inequality we have
\begin{align*}
\|w^{-\epsilon}u_0(t) \|_{L^2} &\leq C(\|\partial_x w^{-\frac{\epsilon}{2}}u_0(t) \|_L^2 + \|w^{-\frac{\epsilon}{2}}u_0(t)\|_{L^2((-1, 1))}) \\
&\leq C \|w^{-\frac{\epsilon}{2}}u(t)\|_{\dot{\mathcal{E}}_{\pm \infty}}
\end{align*}
Where we used the fact that $n\neq 0$ or $m> 0$ to bound $\|w^{-\frac{\epsilon}{2}}u_0(t)\|_{L^2((-1, 1))}$ by $C\left<\Delta_r \left( \frac{P}{\lambda^2(a^2+r^2)^2} + m^2\right)u_0, u_0 \right>$.
Finally we conclude thanks to lemma \ref{propEstGen}.
\end{proof}


\begin{lemma}
For all $u \in \dot{\mathcal{E}}^n$,
\[
W_{-\infty}\Omega_{-\infty}u + W_{+\infty}\Omega_{+\infty}u = u
\]
\end{lemma}
\begin{proof}
Because $W_{\pm \infty}$ and $\Omega_{\pm \infty}$ are defined as a strong limit of uniformly bounded family of operators, we can rewrite

\begin{align*}
W_{-\infty}\Omega_{-\infty}u + W_{+\infty}\Omega_{+\infty}u &= \lim\limits_{t\to +\infty} e^{-it\dot{H}}i_-e^{it\dot{H}_{-\infty}}e^{-it\dot{H}_{-\infty}}i_-e^{it\dot{H}} u + e^{-it\dot{H}}i_+e^{it\dot{H}_{+\infty}}e^{-it\dot{H}_{+\infty}}i_+ e^{it\dot{H}}u \\
&= \lim\limits_{t\to +\infty}e^{-it\dot{H}} i_-^2 e^{it\dot{H}}u + e^{-it\dot{H}}i_+^2 e^{it\dot{H}} u \\
&= u 
\end{align*}

\end{proof}

\begin{remark}
It is not possible to do that directly with $W_{\pm}$ and $\Omega_{\pm}$ because $W_{\pm}$ is not defined as a strong limit of bounded operators.
\end{remark}

\begin{lemma}
\label{infIneq}
There exists a constant $C>0$ such that: for all $u\in \dot{\mathcal{E}}^n$, 
\[
\|u\|_{\dot{\mathcal{E}}} \leq C\left(\|\Omega_{+\infty} u \|_{\dot{\mathcal{E}}_{+\infty}} + \|\Omega_{-\infty} u \|_{\dot{\mathcal{E}}_{-\infty}}\right)
\]
\end{lemma}

\begin{proof}
It is a direct consequence of the previous lemma.
\[
u = W_{-\infty}\Omega_{-\infty}u + W_{+\infty}\Omega_{+\infty}u
\]
We then apply triangular inequality and use the fact that $W_{\pm \infty}$ are bounded.
\end{proof}

\begin{lemma}
\label{finIneq}
We assume that $n\neq 0$ or $m> 0$.
There exists a constant $C>0$ such that, for all $u\in \dot{\mathcal{E}}^{fin,n}_{+\infty}$:
\[
\|u\|_{\dot{\mathcal{E}}_{+\infty}} \leq C\left(\|\Omega_{T,+}^{+\infty}u\|_{\dot{\mathcal{E}}_{T,+}} + S(u)\right)
\]
where 
\[
S(u) := \limsup_{t\to +\infty} \|i_- e^{it\dot{H}_{+\infty}}u\|_{\dot{\mathcal{E}}_+\infty}
\]
\end{lemma}

\begin{remark}
The definition of $S$ makes sense because $i_-$ is bounded from $\dot{\mathcal{E}}^n_{+\infty}$ into itself when $n\neq 0$ or $m> 0$.
\end{remark}
\begin{proof}
Let $u \in \dot{\mathcal{E}}^{fin,n}_{+\infty}$.
We decompose $u = u_{1} + ... + u_{q}$ where $u_{i}$ is an eigenvector of $P$ associated with the eigenvalue $\lambda_{i}$. 
Let $t\in \mathbb{R}$.
By the uniform boundedness of $e^{-it\dot{H}_{+\infty}}$, we have a constant $C'$ independent of $t$ and $u$ such that
\[
\|u\|_{\dot{\mathcal{E}}_{+\infty}} \leq C' \|e^{it\dot{H}_{+\infty}} u \|_{\dot{\mathcal{E}}_{+\infty}}
\]
We then compute:
\begin{align*}
 \|e^{it\dot{H}_{+\infty}} u \|^2_{\dot{\mathcal{E}}_{+\infty}} &\leq 3\|i_+e^{it\dot{H}_{+\infty}}u\|^2_{\dot{\mathcal{E}}_{+\infty}} + 3\|i_-e^{it\dot{H}_{+\infty}} u\|^2_{\dot{\mathcal{E}}_{+\infty}}+ 3 \| (1-(i_- + i_+))e^{it\dot{H}_{+\infty}}u\|^2_{\dot{\mathcal{E}}_{+\infty}}
\end{align*}
Because $(1-(i_- + i_+))$ is compactly supported, 
\[
\| (1-(i_- + i_+))e^{it\dot{H}_{+\infty}}u\|^2_{\dot{\mathcal{E}}_{+\infty}} = o(1)
\]
when $t\to +\infty$ by the lemma \ref{propEstGen}.
Moreover,
\[
\|i_-e^{it\dot{H}_{+\infty}} u\|^2_{\dot{\mathcal{E}}_{+\infty}} \leq S(u)^2 + o(1)
\]
just by definition of the limsup.
We write for simplicity $u(t) := e^{it\dot{H}_{+\infty}}u$ (we can also write $u_k(t):= e^{it\dot{H}_{+\infty}}u_k$ and because $\dot{H}_{+\infty}$ commute with $P$ it is still an eigenvector of $P$. It is where we need separability. )
We now analyze the term
\begin{align*}
\|i_+e^{it\dot{H}_{+\infty}}u\|^2_{\dot{\mathcal{E}}_{+\infty}} &= \|\partial_x i_+ u(t)_0 \|^2_{L^2} + \sum_{k=1}^{q} \lambda_k \left<\frac{\Delta_r}{\lambda^2(a^2+r^2)^2}i_+ u_k(t)_0 , i_+ u_k(t)_0\right> \notag \\
&\phantom{{}={}} + \left<\Delta_r m^2 i_+u(t)_0,i_+u(t)_0\right> + \| u(t)_1 + l_+u(t)_0 \|^2_{L^{2}} \\
&\leq 2\|(\partial_x + i(l-l_+)) i_+ u(t)_0 \|^2_{L^2} + 2\|(l-l_+) i_+ u(t)_0 \|^2_{L^2}\notag \\
&\phantom{{}={}} + \sum_{k=1}^{q} \lambda_k \left<\frac{\Delta_r}{\lambda^2(a^2+r^2)^2}i_+ u_k(t)_0 , i_+ u_k(t)_0\right>\notag \\
&\phantom{{}={}} + \left<\Delta_r m^2 i_+u(t)_0,i_+u(t)_0\right> + \| u(t)_1 + l_+u(t)_0 \|^2_{L^{2} }
\end{align*}
In this expression, 
\[ 2\|(\partial_x + i(l-l_+)) i_+ u(t)_0 \|^2_{L^2} + \| u(t)_1 + l_+u(t)_0 \|^2_{L^{2}} \] 
can be controlled by 
\[ \|i_+ u(t)\|^2_{\dot{\mathcal{E}}_{T,+}} = \| e^{-itH_{T,+}}i_+ u(t) \|^2_{\dot{\mathcal{E}}_{T,+}}\]
and the other terms 
\[ 2\|(l-l_+) i_+ u(t)_0 \|^2_{L^2} + \sum_{k=1}^{q} \lambda_k \left<\frac{\Delta_r}{\lambda^2(a^2+r^2)^2}i_+ u_k(t)_0 , i_+ u_k(t)_0\right> + \left<\Delta_r m^2 i_+u(t)_0,i_+u(t)_0\right>
\]
converge to zero when $t\to +\infty$ by the lemma \ref{propFirstComp} (propagation of the $L^2$ norm of the first component). Finally, we take $t\to +\infty$ in the bound and we get
\[
\|u\|_{\dot{\mathcal{E}}_{+\infty}} \leq C(\|\Omega_+ u\|_{\dot{\mathcal{E}}_{T,+}}+ S(u))
\]
\end{proof}

Now we want somehow extend the inequality by density. That is why we need to prove the continuity of $S$. This is the goal of the two next lemmas:
\begin{lemma}
Let $n\neq 0$ or $m> 0$.
There exists a constant $C>0$ such that:
For all $u,v\in \dot{\mathcal{E}}^n_{+\infty}$, 
\begin{align}
S(u+v) &\leq S(u)+S(v) \label{inegTri}\\
S(u) &\leq C\|u\|_{\dot{\mathcal{E}}_{+\infty}} \label{borne}
\end{align}
\end{lemma}

\begin{proof}
Let's prove \eqref{inegTri}:
\begin{align*}
S(u+v) &= \limsup_{t\to +\infty} \|i_-e^{it\dot{H}_{+\infty}}(u+v) \|_{\dot{\mathcal{E}}_{+\infty}}\\
& \leq \limsup_{t\to +\infty} \left( \|i_-e^{it\dot{H}_{+\infty}}u \| + \|i_-e^{it\dot{H}_{+\infty}}v\| \right) \\
& \leq \limsup_{t\to +\infty} \|i_-e^{it\dot{H}_{+\infty}}u \| + \limsup_{t\to +\infty} \|i_-e^{it\dot{H}_{+\infty}}v \| \\
&\leq S(u) + S(v)
\end{align*}
Finally, we prove \eqref{borne}:
\begin{align*}
\|i_-e^{it\dot{H}_{+\infty}}u \|_{\dot{\mathcal{E}}_{+\infty}} &\leq C'\|e^{it\dot{H}_{+\infty}}u \|_{\dot{\mathcal{E}}_{+\infty}} \\
&\leq C\|u\|_{\dot{\mathcal{E}}_{+\infty}}
\end{align*}
because $i_-$ is a bounded operator from $\dot{\mathcal{E}}_{+\infty}$ into itself and $e^{it\dot{H}_{+\infty}}$ is uniformly bounded on its energy space.
\end{proof}

\begin{lemma}
$S:\dot{\mathcal{E}}^n_{+\infty} \rightarrow [0, +\infty)$ is Lipschitz.
\end{lemma}

\begin{proof}
It is a consequence of the previous lemma. Let $u,v \in \dot{\mathcal{E}}^n_{+\infty}$.
\begin{align*}
S(u) - S(v)&\leq S(u-v) \text{     by \eqref{inegTri}}\\
&\leq C\|u-v\|_{\dot{\mathcal{E}}_{+\infty}} \text{     by \eqref{borne}} \\
\end{align*}
and 
\begin{align*}
S(v)-S(u) &\leq S(v-u) \\
& \leq C\|u-v\|_{\dot{\mathcal{E}}_{+\infty}}
\end{align*}
\end{proof}

Now we can deduce by density the following proposition from lemma \ref{finIneq}
\begin{prop}
\label{genIneq}
We assume that $n\neq 0$ or $m> 0$.
There exists a constant $C>0$ such that, for all $u\in \dot{\mathcal{E}}^n_{+\infty}$:
\[
\|u\|_{\dot{\mathcal{E}}_{+\infty}} \leq C\left(\|\Omega_{T,+}^{+\infty}u\|_{\mathcal{E}_{T,+}} + S(u)\right)
\]
\end{prop}

The last thing we need to prove is
\begin{lemma}
\label{SRanOmega}
 For all $u\in Ran(\Omega_{+\infty})$,
\[
S(u) = 0
\]
\end{lemma}

\begin{proof}
Let $u\in Ran(\Omega_{+\infty})$. There exists $v\in \dot{\mathcal{E}}^n$ such that (in $\dot{\mathcal{E}}^n_{+\infty}$):
\[ u= \lim\limits_{t\to +\infty} e^{-it\dot{H}_{+\infty}}i_+ e^{it\dot{H}}v \]
We write \[ u = e^{-it\dot{H}_{+\infty}}i_+ e^{it\dot{H}}v + \epsilon(t) \]
whith $\lim\limits_{t \to +\infty}\| \epsilon(t) \|_{\dot{\mathcal{E}}_{+\infty}} = 0$.
Let $t\in \mathbb{R}$.
By uniform boundedness of $i_-e^{it\dot{H}^n_{+\infty}}$ on $\dot{\mathcal{E}}^n_{+\infty}$:
\begin{align*}
\| i_- e^{it\dot{H}_{+\infty}} u \|_{\dot{\mathcal{E}}_{+\infty}} &\leq \|i_-i_+ e^{it\dot{H}}v\|_{\dot{\mathcal{E}}_{+\infty}} + C\|\epsilon(t)\|_{\dot{\mathcal{E}}_{+\infty}}
\end{align*}
The limit of the right hand side when $t\to +\infty$ is zero by lemma \ref{propEstGen}.
\end{proof}

Finally we can use together proposition \ref{genIneq}, lemma \ref{SRanOmega} and lemma \ref{infIneq} to prove the following proposition, which was the goal of this subsection.
\begin{prop}
We assume $n\neq 0$ or $m\neq 0$.
Then, there exists $C>0$ such that
For all $u\in \dot{\mathcal{E}}$
\[
\|u\|_{\dot{\mathcal{E}}}\leq C(\|\Omega_{T,+} u \|_{\dot{\mathcal{E}}_{T,+}} + \|\Omega_{T,-} u\|_{\dot{\mathcal{E}}_{T,-}})
\]
In particular $\Omega$ is injective.
\end{prop}
\begin{proof}
Let $u\in \dot{\mathcal{E}}$, by lemma \ref{infIneq} we get
\[ \|u\|_{\dot{\mathcal{E}}} \leq C\left(\|\Omega_{+\infty}u\|_{\dot{\mathcal{E}}_{+\infty}}+\|\Omega_{-\infty}u\|_{\dot{\mathcal{E}}_{-\infty}}\right) \]
Then by lemma \ref{genIneq} we have
\[ \|\Omega_{+\infty}u\|_{\dot{\mathcal{E}}_{+\infty}}\leq C\left(\|\Omega^{+\infty}_{T,+}\Omega_{+\infty}u\|_{\dot{\mathcal{E}}_{T,+}} + S(\Omega_{+\infty}u)\right) \]
and by lemma \ref{SRanOmega}, $S(\Omega_{+\infty}u) = 0$. Corresponding results for indices $-$ enable to conclude the proof.
\end{proof}

\section{Main theorem}

\subsection{Trace operator}
By Leray's theorem, for $u^{init} \in (C^{\infty}_0(\Sigma_0))^2\cap (\mathcal{Y}^n)^2$, there exists a unique solution in $C^{\infty}_0(\mathcal{M})$ to the equation:
\[
\begin{cases}
(\square + m^2)u = 0 \\
u(0) = u^{init}_0 \\
\frac{1}{i} \partial_t u(0) = u^{in}_1
\end{cases}
\]
Moreover, this solution extends smoothly to the horizons. To see that, we can take a spacelike hypersurface $\tilde{\Sigma}_0$ (in the maximal extension of Kerr De Sitter space-time, see \cite{borthwick}) extending $\Sigma_0$ such that its domain of dependence contains the horizon (we can proceed one horizon at a time and use the kruskal domains). Then by applying Leray theorem, we find a smooth extension of our solution to an open set containing the horizon.
In particular, we can define the following trace operators
\begin{definition}
\[
\mathcal{T}_{\pm}:
\begin{cases}
(C^{\infty}_0)^2\cap (\mathcal{Y}^n)^2 \rightarrow C^{\infty}(\mathfrak{H}^{future}_{\pm}) \\
u^{init} \mapsto u_{|\mathfrak{H}^{future}_{\pm}}
\end{cases}
\]
\end{definition}
\begin{remark}
Note that on $\mathfrak{H}^{future}_{\pm}$, $u$ and $D_t u$ are not independent since $\partial_t$ is tangent to $\mathfrak{H}^{future}_{\pm}$. Therefore, we do not need to take the trace of $D_t u$.
\end{remark}
\begin{remark}
Because $\partial_\phi$ is a Killing vector field, $D_\phi - n$ commutes with $\square_g + m^2$. Then, if $u$ is a smooth solution of the Klein-Gordon equation, $(D_\phi - n)u$ is also a solution. In particular, if the initial data belong to $(\mathcal{Y}^n)^2$, $(D_\phi -n)u = 0$ on the whole domain of dependance.
\end{remark}
\begin{remark}
We could do the same construction for the past horizon but in this work we focus on the future horizon.
\end{remark}
We emphasize that, thanks to the finite speed of propagation, we can apply our two lemmas \ref{continuity} and \ref{uniqueness} to prove that for all $t\in \mathbb{R}$:
\[
e^{it\dot{H}^n}u^{init} = \begin{pmatrix} u(t) \\ D_t u(t) \end{pmatrix}
\] 

\subsection{Energy spaces on the horizons}
We begin by defining natural ways to identify $\Sigma_0$ with the horizons.
\begin{definition}
We define \[\mathfrak{F}_{+}:
\begin{cases}
\Sigma_0 \rightarrow  \mathfrak{H}^{future}_+ \\
(0, r, \theta, \phi) \mapsto (-T(r), r_+, \theta, \phi - A(r))_{{}^*Kerr}
\end{cases} \]
which maps the intersection of an outgoing principal null geodesic with $\Sigma_0$ to its intersection with $\mathfrak{H}^{future}_+$.
Similarly we define $\mathfrak{F}_{-}$.
\end{definition}
\begin{prop}
$\mathfrak{F}_{\pm}$ are diffeomorphisms.
\end{prop}
\begin{proof}
The explicit expression shows that $\mathfrak{F}_{\pm}$ are $C^{\infty}$. We can also find the $C^{\infty}$ inverse:
 \[(\mathfrak{F}_{+})^{-1} = 
\begin{cases}
\mathfrak{H}^{future}_+\rightarrow \Sigma_0 \\
({}^*t, r_+, \theta, {}^*\phi) \mapsto (0, T^{-1}(-{}^*t), \theta, {}^*\phi + A(T^{-1}(-{}^*t)))
\end{cases} \]
and similarly for $\mathfrak{F}_{-}$
\end{proof}
\begin{definition}
\[
\mathcal{F}_{-}:
\begin{cases}
C^{\infty}_0(\mathfrak{H}^{future}_{-})\cap Ker(\partial_{\phi^{*}} - n) \rightarrow \mathcal{E}^{l}_{T,-} \\
u \mapsto \begin{pmatrix}(\mathfrak{F}_{-})^* u \\ i w_{-}(\mathfrak{F}_{-})^* u\end{pmatrix}
\end{cases}
\]
\[
\mathcal{F}_{+}:
\begin{cases}
C^{\infty}_0(\mathfrak{H}^{future}_{+})\cap Ker(\partial_{{}^*\phi} - n) \rightarrow \mathcal{E}^{r}_{T,+} \\
u \mapsto \begin{pmatrix}(\mathfrak{F}_{+})^* u \\ i w_{+}(\mathfrak{F}_{+})^* u\end{pmatrix}
\end{cases}
\]
\end{definition}
\label{energyHorizon}
We also define energy spaces on the horizon $\mathfrak{H}^{future}_-$ (resp. $\mathfrak{H}^{future}_+$) by transporting the norm on $\mathcal{E}^l_{T,-}$ (resp. $\mathcal{E}^r_{T,+}$). More explicitly, we define for $u \in C^{\infty}_0(\mathfrak{H}^{future}_+)\cap Ker(\partial_{{}^*\phi} - n)$ and $v \in C^{\infty}_0(\mathfrak{H}^{future}_-)\cap ker(\partial_{\phi^*} -n)$
\begin{align*}
\|\phi \|_{\mathcal{E}^n_{\mathfrak{H}_+}} &:= \left\|\mathcal{F}_{+} u\right\|_{\dot{\mathcal{E}}_{T,+}}\\
\|\phi \|_{\mathcal{E}^n_{\mathfrak{H}_-}} &:= \left\|\mathcal{F}_{-} v\right\|_{\dot{\mathcal{E}}_{T,-}}
\end{align*}
We define $\mathcal{E}^n_{\mathfrak{H}_{+}}$ (resp. $\mathcal{E}^n_{\mathfrak{H}_{-}}$)  by completion of $C^{\infty}_0(\mathfrak{H}^{future}_{+})\cap Ker(\partial_{{}^*\phi} - n)$ (resp. $C^{\infty}_0(\mathfrak{H}^{future}_{-})\cap Ker(\partial_{\phi^{*}} - n)$) for the corresponding norm. Then, $\mathcal{F}_{+/-}$ extend to surjective isometries from $\mathcal{E}^n_{\mathfrak{H}_{+/-}}$ to $\mathcal{E}^{r/l}_{T,+/-}$.
\begin{remark}
We can compute explicitly the norms $\|.\|_{\mathcal{E}^n_{\mathfrak{H}_\pm}}$ and we find, for $u \in C^{\infty}_0(\mathfrak{H}^{future}_{+})\cap Ker(\partial_{{}^*\phi} - n)$ and $v\in C^{\infty}_0(\mathfrak{H}^{future}_{-})\cap Ker(\partial_{\phi^{*}} - n)$
\begin{align*}
\|u\|^2_{\mathcal{E}^n_{\mathfrak{H}_{+}}} &= 2\int_{\mathfrak{H}^{future}_{+}} \left|\partial_{{}^*t} u + \frac{ian}{a^2+r_+^2}u\right|^2\dd {}^*t \dd {}^*\omega \\
\|v\|^2_{\mathcal{E}^n_{\mathfrak{H}_{-}}} &= 2\int_{\mathfrak{H}^{future}_{-}} \left|\partial_{t^*} u + \frac{ian}{a^2+r_-^2}u \right|^2 \dd t^* \dd \omega^*
\end{align*}
If $u$ is the trace of a function defined on $\bar{\mathcal{M}}$, $\frac{1}{2}\|u\|_{\mathcal{E}^n_{\mathfrak{H}_{+}}}$ correspond to the flux of the contraction of $T(u)$ with $X$ through $\mathfrak{H}^{future}_+$ as mentioned in section \ref{mainRes}.
\end{remark}
\begin{definition}
We define the map
\[
\mathcal{F}:
\begin{cases}
\mathcal{E}^n_{\mathfrak{H}_{-}}\oplus \mathcal{E}^n_{\mathfrak{H}_+} \rightarrow \mathcal{P}\\
(u^l, u^r) \mapsto \left( \mathcal{F}_{-} u^l, \mathcal{F}_{+} u^r \right)
\end{cases}
\]
This application is a surjective isometry between $\mathcal{P}$ and the profile subspace
\[
\mathcal{P}_{\mathfrak{H}} := \mathcal{E}^n_{\mathfrak{H}_{-}}\oplus \mathcal{E}^n_{\mathfrak{H}_+}
\]
\end{definition}

\subsection{Proof of the main theorem}
In this section, we will show the following theorem:

\begin{theorem}
\label{mainTheo}
For all $u \in (C^{\infty}_0(\Sigma_0) \cap \mathcal{Y}^n)^2$ ,
\begin{align*}
\mathcal{F}^{-1} \Omega u &= (\mathcal{T}_- u, \mathcal{T}_+ u)
\end{align*}
Therefore, the trace operator $\mathcal{T}$ extends uniquely as a bounded operator from $\dot{\mathcal{E}}^n$ to $\mathcal{P}_{\mathfrak{H}}$ (this extension is in fact $\mathcal{F}^{-1} \Omega$). Moreover, this operator is invertible.
\end{theorem}

\begin{lemma}
\label{transport2}
We define the operator valued matrix
\[
W_+ := \begin{pmatrix} w_+ & 0 \\ 0 & w_+ \end{pmatrix}
\]
For $u\in (C^{\infty}_0(\Sigma_0)\cap \mathcal{Y}^n)^2$,
\[
\|e^{-tW_+ } u \|_{\dot{\mathcal{E}}_{T,+}} = \| u\|_{\dot{\mathcal{E}}_{T,+}}
\]
So $e^{-W_+ t}$ extends to a unitary $C^0$-semi group on $\dot{\mathcal{E}}^n_{T,+}$.
Moreover, for $u\in \mathcal{E}^r_{T,+}$,
\[
e^{itH_{T,+}} u = e^{-tW_+} u
\]
\end{lemma}

\begin{proof}
To prove the first assertion, we use the fact that $e^{tw_+}$ commutes with $h_{0,T,+}$ and is unitary on $L^2$.
To prove the second part, we see that the equality is true for $u\in (C^{\infty}_0)^2\cap \mathcal{E}^r_{T,+}$ (by the Kirchoff formula) and we use a density argument (lemma \ref{densityLR}) to conclude.
\end{proof}

An other lemma will be useful to emphasize the link between $e^{itH_{T,+}}$ and $e^{-tW_+}$
\begin{lemma}
\label{comparaisonTransport}
For all $u \in (C^{\infty}_0(\Sigma_0) \cap \mathcal{Y}^n)^2$, we have the following limit in $\dot{\mathcal{E}}^n_{T,+}$:
\[ \lim\limits_{t\to +\infty} e^{tW_+}i_+e^{it\dot{H}}u - e^{-it\dot{H}_{T,+}}i_+e^{it\dot{H}}u = 0 \]
In other words in $\dot{\mathcal{E}}^n_{T,+}$:
\[
\Omega_+ u = \lim\limits_{t \to +\infty}e^{tW_+}i_+e^{it\dot{H}}u
\]
\end{lemma}

\begin{proof}
Let $u \in (C^{\infty}_0(\Sigma_0) \cap \mathcal{Y}^n)^2$.
We use the strong limit property and the uniform boundedness of $e^{it\dot{H}_{T,+}}$ on $\mathcal{E}_{T,+}$ to replace:
\[ i_+e^{it\dot{H}}u = e^{it\dot{H}_{T,+}}\Omega_+ u + \epsilon(t) \]
where $\lim\limits_{t\to +\infty} \epsilon(t) = 0$ in $\dot{\mathcal{E}}^n_{T,+}$.
Finally, using that $Ran(\Omega_+) \subset \mathcal{E}^r_{T,+}$, we have
\[
e^{itH_{T,+}}\Omega_+ u= e^{-tW_+}\Omega_+ u
\]
We get:
\begin{align*}
e^{tW_+}i_+e^{it\dot{H}}u - e^{-itH_{T,+}}i_+e^{it\dot{H}}u &= e^{tW_+}\epsilon(t) - e^{-it\dot{H}_{T,+}}\epsilon(t)
\end{align*}
By the uniform boundedness of $e^{tW_+}$ and $e^{-it\dot{H}_{T,+}}$, we obtain that the limit is zero in $\dot{\mathcal{E}}^n_{T,+}$ and then:
 \[ \lim\limits_{t\to +\infty} e^{tW_+}i_+e^{it\dot{H}}u = \Omega_+ u \]
\end{proof}

\begin{proof}[Proof of the theorem]
Let $u \in (C^{\infty}_0(\Sigma_0) \cap \mathcal{Y}^n)^2$.
We have to prove the two following equalities
\begin{align*}
 \mathcal{F}_+^{-1} \Omega_+ u &= \mathcal{T}_+ u \\
\mathcal{F}_-^{-1} \Omega_- u &= \mathcal{T}_- u
\end{align*}
We only prove the $+$ case, the other being similar.
We denote by $\tilde{u}$ the solution to the Klein-Gordon equation on $\bar{\mathcal{M}}$ and $I_+: (t,r,\theta, \phi)\mapsto i_+(x(r))$ which has a smooth extension to $\bar{\mathcal{M}}$. For simplicity, we still call $I_+$ this extension. Note that $I_{+|_{\mathfrak{H}^{future}_+}} = 1$. 

We saw (thanks to Leray's theorem) that $i_+e^{it\dot{H}^n}u \in C^{\infty}(\mathbb{R}_x \times \mathbb{S}^2)\cap \dot{\mathcal{E}}^n_{T,+}$. We can compute explicitly
\begin{align*}
e^{tW_+}i_+ e^{it\dot{H}}u (x,\theta, \phi)&= i_+(x+t) (e^{it\dot{H}}u) (x+t, \theta, \phi - A(T^{-1}(x)) + A(T^{-1}(x+t))) \\
&= \left(I_+\begin{pmatrix}\tilde{u} \\D_t \tilde{u}\end{pmatrix}\right)(-x, T^{-1}(x+t), \theta, \phi - A(T^{-1}(x)))_{{}^*Kerr}
\end{align*}
The first equality comes from the explicit action of $e^{tW_+}$ on $C^{\infty}(\R_x\times\mathbb{S}^2)\cap \dot{\mathcal{E}}^n_{T,+}$ and the second line is the reformulation in ${}^*Kerr$ coordinates.
By letting $t\to +\infty$ in the previous equality, we have the pointwise limit:
\begin{align*}
\lim\limits_{t\to +\infty}e^{tW_+}i_+ e^{it\dot{H}}u (x,\theta, \phi) &= \begin{pmatrix}\tilde{u} \\D_t \tilde{u}\end{pmatrix}(-x, r_+, \theta, \phi-A(T^{-1}(x)))_{{}^*Kerr} \\
&=  \begin{pmatrix}(\mathfrak{F}_+)^*\mathcal{T}_+ u \\ (\mathfrak{F}_+)^* D_t \mathcal{T}_+ u\end{pmatrix}(x, \theta, \phi)
\end{align*}
It is only a pointwise limit but we can recover the limit in $\dot{\mathcal{E}}_{T,+}$ by checking in the same way that we have the pointwise convergences:
\begin{align*}
\lim\limits_{t\to +\infty}(\partial_x + i(l-l_+))(e^{tW_+}i_+ e^{it\dot{H}} u)_0 &= (\partial_x + i(l-l_+)) (\mathfrak{F}_+)^* \mathcal{T}_+ u \\
\lim\limits_{t\to +\infty} (e^{tW_+}i_+e^{it\dot{H}} u)_1 + l_+ (e^{tW_+}i_+e^{it\dot{H}} u)_0 &= (\mathfrak{F}_+)^* D_t\mathcal{T}_+ u + l_+ (\mathfrak{F}_+)^* \mathcal{T}_+ u
\end{align*}
Adding the fact that $(\partial_x + i(l-l_+))(e^{tW_+}i_+ e^{it\dot{H}} u)_0$ and $(e^{tW_+}i_+e^{it\dot{H}} u)_1 + l_+ (e^{tW_+}i_+e^{it\dot{H}} u)_0$ have a limit in $L^2$ (see lemma \ref{comparaisonTransport}), we deduce that 
$\begin{pmatrix}(\mathfrak{F}_+)^* \mathcal{T}_+ u \\ (\mathfrak{F}_+)^* D_t\mathcal{T}_+ u\\ \end{pmatrix}\in \dot{\mathcal{E}}^n_{T,+}$ and for the topology of $\dot{\mathcal{E}}^n_{T,+}$,
\[
\lim\limits_{t\to +\infty} e^{tW_+}i_+e^{it\dot{H}} u = \begin{pmatrix}(\mathfrak{F}_+)^* \mathcal{T}_+ u \\ (\mathfrak{F}_+)^* D_t\mathcal{T}_+ u\end{pmatrix}
\]
Then 
\[
\Omega_+ u = \begin{pmatrix}(\mathfrak{F}_+)^* \mathcal{T}_+ u \\  (\mathfrak{F}_+)^* D_t\mathcal{T}_+ u \end{pmatrix}
\]
The fact that $\Omega_+ u \in \mathcal{E}^r_{T,+}$ enables to show that $\left\| \mathcal{T}_+ u\right\|_{\mathcal{E}^n_{\mathfrak{H}_+}}< +\infty$ (the norm make sense on a smooth function). Then $\mathcal{T}_+ u$ can be identified with an element of $\mathcal{E}^n_{\mathfrak{H}_+}$ and $\mathcal{F}_+ \mathcal{T}_+ u = \Omega_+ u$.
We conclude the proof of the equality by applying $\mathcal{F}_+^{-1}$ to both sides.

From there, we deduce that $\mathcal{F}^{-1} \Omega$ is an extension of the trace and such an extension is unique by density of $(C^{\infty}_0\cap \mathcal{Y}^n)^2$ in $\dot{\mathcal{E}}^n$. The extension is invertible because we know explicitly the inverse of $\mathcal{F}^{-1} \Omega$ (it is $W \mathcal{F}$ by theorem \ref{inversion}).
\end{proof}

\begin{appendix}
\section{Appendix}
\subsection{General facts}
In this section, we state general elementary facts which can be understood independently of the other parts of this work, but are needed in the proofs.
The first result is about decomposition of Hilbert spaces. 
\begin{lemma}
\label{decompo}
Let $\mathcal{H}$, $\mathcal{H}_1$ and $\mathcal{H}_2$ be separable hilbert spaces such that $\mathcal{H}=\mathcal{H}_1\otimes\mathcal{H}_2$ (in the Hilbert sense). Let $A$ be a selfadjoint unbounded operator on $\mathcal{H}$ and assume that there exists $(e_i)_{i\in \mathbb{N}}$ a hilbert basis of $\mathcal{H}_2$ such that  $A$ commutes with all the orthogonal projections $\Pi_{i}$ on the closed subspaces $\mathcal{H}_1\otimes \C e_i$. Then: 
\begin{itemize}
\item For all $i\in \mathbb{N}$, $A$ induces selfadjoint operators $A_i$ on $\mathcal{H}_1$ by the unitary identification
\[
\phi_i: 
\begin{cases}
\mathcal{H}_1 \rightarrow \mathcal{H}_1 \otimes \C e_i \\
u \mapsto u\otimes e_i
\end{cases}
\]
\item For all borelian function $f$ on $\R$, 
$Dom(f(A)) = \oplus_{i\in \mathcal{N}} Dom(f(A_i))\otimes \C e_i$ (in the hilbert sense) where $Dom(f(A))$ and the $Dom(f(A_i))$ are endowed with the graph norm. In particular, if all the $A_i$ are equal to some $\tilde{A}$, $Dom(f(A)) = Dom(f(\tilde{A}))\otimes \mathcal{H}_2$(in the Hilbert sense).
\end{itemize}
\end{lemma}

\begin{proof}
The first point follows from the Stone theorem ($e^{itA}$ commutes with $\Pi_i$ then it induces a $C^0$-semigroup of isometries on $\mathcal{H}_1\otimes \C e_i$ and the generator is exactly the operator induced by $A$). 
We now prove the second part. We recall that by construction, the functional calculus and the induction on a closed subspace commutes together. Then $f(A_i)$ is the induction of $f(A)$ on $\mathcal{H}_1\otimes \C e_i$ (through $\phi_i$) and $\phi_i(Dom(f(A_i))) = Dom(f(A_i))\otimes \C e_i = Dom(f(A))\cap \mathcal{H}_1\otimes \C e_i$. Let $u\in Dom(f(A))$ then the sequence $u_n = \sum_{k=0}^n \Pi_k(u)$ converges towards $u$ in $\mathcal{H}$. Because $\Pi_k$ commutes with $f(A)$ we also have that $f(A)u_n = \sum_{k=0}^n \Pi_k (f(A)u)$ converges towards $f(A)u$ in $\mathcal{H}$. We deduce that the isometric (when $Dom(f(A))$ and the $Dom(f(A_i))$ are endowed with the graph norm) inclusion $\oplus_{i\in \mathcal{N}} Dom(f(A_i))\otimes \C e_i \hookrightarrow Dom(f(A))$ has a dense range. The range is also closed (because complete by the isometry property) and we conclude the proof of the lemma.
\end{proof}

\begin{remark}
If $A$ is only assumed to be closed, then we can also define the closed operators $A_i$ in the same way and the same argument gives $Dom(A) = \oplus_{i\in \mathcal{N}} Dom(A_i)\otimes \C e_i$ (in the hilbert sense) where $Dom(A)$ and the $Dom(A_i)$ are endowed with the graph norm.
\end{remark}

The two following lemmas are useful to establish an explicit expression for a dynamics generated by a selfadjoint operator.
We begin by recalling an elementary uniqueness result for the solution of an evolution equation.
\begin{prop}
\label{uniqueness}
Let $U(t)_{t\in \R}$ be a $C^0$-group of bounded operators on a hilbert space $\mathcal{H}$ and let $A$ be its infinitesimal generator. Then if $u \in C^{1}(\R, \mathcal{H})$ verifies:
\begin{itemize}
\item $u(0) \in Dom(\mathcal{A})$
\item $\forall t \in \R, u\in Dom(A) \text{  and  } \partial_t u= Au$
\end{itemize}
then, $u = U(t)u(0)$.
\end{prop}

We do not give a proof here. We refer to \cite{Pazy}, section 4.1 on the homogeneous Cauchy problem for more details.

To apply this proposition, we often need to show that an explicit function is in $C^{1}(\R, L^2)$ and the following technical lemma is useful.

\begin{lemma}
\label{continuity}
Let $X$ be a topological space equipped with a radon measure (borelian locally finite) $\mu$.
Let $u\in C^{1}(\R\times X)$ satisfy the following condition: \newline
for all $t\in \R$, there exists some $\epsilon>0$ such that $supp(1_{[t-\epsilon, t+\epsilon]}u)$ is compact. \newline
Then $u\in C^1(\R, L^2(X,\mu))$ and $\frac{\dd}{\dd t}u = \partial_t u$.
\end{lemma}


\begin{proof}
It is a corollary of the Lebesgue theorem (and the condition on $u$ ensures the domination hypothesis). Let $t \in \R$, then for $h\in(-\epsilon, \epsilon)$, we have $|\frac{u(t+h,x)- u(t)}{h}|^2 \leq \sup_{supp(1_{[t-\epsilon, t+\epsilon]}u)}|\partial_t u|^2 1_{[t-\epsilon, t+\epsilon]}u$ which is well defined by continuity of $\partial_t u$ and integrable because $\mu$ is locally finite. So by Lebesgue theorem, we can pass to the limit and we find $\lim\limits_{h\to 0}\frac{u(t+h,.)-u(t,.)}{h} = \partial_t u(t,.)$ in $L^2(X,\mu)$. For the continuity of the derivative, we use again the Lebesgue theorem. 
\end{proof}

\begin{lemma}[Hardy type inequality]
\label{HardyIneq}
Let $q\in L^1_{loc}(\mathbb{R})$ be exponentially decaying at $\pm \infty$. Then, there exists $C>0$ such that for all $u \in C^{\infty}_0(\R)$ :
\[
\|q u \|_{L^2}\leq C \left(\|\partial_x u\|_{L^2} + \|u\|_{L^2(-1,1)}\right)
\] 
\end{lemma}
\begin{proof}
We use the equalities $u(x) = u(0)+\int_{0}^x \partial_x u(s) \dd s$ and $2u(0) = \int_{-1}^1 \left(u(s) - \int_{0}^s \partial_x u(t) \dd t \right)\dd s$ to get 
\begin{align*}
\int_{\R} |q(x)|^2 |u(x)|^2\dd x \leq \int_{\R} 3|q(x)|^2 \left(\|u\|_{L^2(-1,1)}^2 + |x|\|\partial_x u\|^2_{L^2} + \|\partial_x u \|^2_{L^2(-1,1)}\right)\dd x
\end{align*}
We conclude by integrability of $|x||q(x)|^2$ and $|q(x)|^2$.
\end{proof}
\begin{remark}
We often use a density argument to get this bound for $u$ in a bigger space.
\end{remark}

\subsection{Density lemmas}
The following lemmas are useful to recover general results from computations on some particular set of functions.
\begin{lemma}
\label{density}
We denote by $Z^{fin}\subset C^{\infty}(\mathbb{S}^2)$ the set of finite sums of eigenfunctions of $P$ (i.e. $Z^{fin}:= \oplus_{q\in N}Z_q$ where the direct sum is taken in the sense of vector spaces).
Then $C^{\infty}_0(\R) \otimes Z^{fin}$ (in the vector space sense) is dense in $L^2(\R \times \mathbb{S}^2)$, $\mathcal{H}^1$, $\mathcal{H}^2$, $Dom\left(h^{\frac{1}{2}}_{0,\pm \infty}\right)$, $Dom\left(h^{\frac{1}{2}}_{0, T, \pm}\right)$, $Dom(h_{0,\pm \infty})$ and $Dom(h_{0, T, \pm})$ (endowed with the graph norms).
Moreover, we have a similar result for the space of data with fixed angular momentum $n$:
$C^{\infty}_0(\R) \otimes (Z^{fin}\cap Y^n)$ is dense in $\mathcal{Y}^n$, $\mathcal{H}^1_n$, $\mathcal{H}^2_n$, $Dom\left(h^{\frac{1}{2}}_{0,\pm \infty}\right)\cap \mathcal{Y}^n$, $Dom\left(h^{\frac{1}{2}}_{0, T, \pm}\right)\cap \mathcal{Y}^n$, $Dom(h_{0,\pm \infty})\cap \mathcal{Y}^n$ and $Dom(h_{0, T,\pm})\cap \mathcal{Y}^n$.
\end{lemma}
\begin{proof}
We can use the lemma \ref{decompo} with $\mathcal{H}_1 = L^2(\R)$, $\mathcal{H}_2 = L^2(\mathbb{S}^2)$, $\mathcal{H} = L^2(\R \times \mathbb{S}^2)$, a basis of eigenfunction of $P$ $(e_i)$ and for the selfadjoint operator A we take $Id$, $h_{0, T, \pm}$ and $h_{0,\pm \infty}$ which verify the commutation property. Then for $s \in \left\{\frac{1}{2}, 1 \right\}$, $Dom(A^s) = \oplus_{i\in\N}Dom(A_i^s)\otimes \C e_i$. We conclude by density of $C^{\infty}_0(\R)$ in $Dom(A_i^s)$ for the graph norm (by a standard convolution and cutoff argument).
We now prove the result for spaces of fixed angular momentum $n$. We have that $P$ commutes with the orthogonal projection $\Pi_n$ on $Y^n$. We deduce that every function of $P$ and in particular the orthogonal projections on the eigenspaces $\Pi_{e_i}$ on $L^2(\R)\otimes\C e_i$ commutes with $\Pi_n$. Then $\Pi_{e_i}$ preserves $\mathcal{Y}^n$ so for $u$ in $Dom(A)\cap \mathcal{Y}^n$, $\sum_{k=0}^N \Pi_{e_i}(u) \in \mathcal{Y}^n$ and converges towards $u$ in the graph norm by the lemma).
\end{proof}

A useful consequence of this lemma is the following
\begin{lemma}
\label{density2}
$(C^{\infty}_0(\R) \otimes (Z^{fin}\cap Y^n))^2$ is dense in $\mathcal{E}^n_{T,\pm}$ (hence in $\dot{\mathcal{E}}^n_{T,\pm}$ by continous and dense inclusion), in $\mathcal{E}^n_{\pm \infty}$ (hence in $\dot{\mathcal{E}}^n_{\pm \infty}$), in $Dom(H_{T,\pm})$ and in $Dom(H^n_{\pm \infty})$ for the graph norm (hence in $Dom(\dot{H}_{T,\pm})$ and $Dom(\dot{H}^n_{\pm \infty})$ by continuous and dense inclusion).
\end{lemma}
\begin{proof}
We use the definition of the space (for example $\mathcal{E}^n_{T,+} = \mathcal{H}^1_n \oplus \mathcal{Y}^n$) and the lemma \ref{density} gives  the density. For $Dom(H_{T,\pm})$ (resp. $Dom(H_{\pm \infty})$) we also use the fact that the graph norm is equivalent to the norm on $\mathcal{H}^2\oplus\mathcal{H}^1$ (resp. on $Dom(h_{0,\pm\infty})\oplus Dom\left(h_{0,\pm \infty}^{\frac{1}{2}}\right)$).
\end{proof}

\begin{lemma}
\label{density3}
The space $\mathcal{D}^{fin}_{T,\pm}:= (C^{\infty}_0(\R_x \times \mathbb{S}^2))^2 \cap \mathcal{E}^{fin,n,L}_{T,\pm}$ is dense in $\mathcal{E}^n_{T, \pm}$ (and thus in $\dot{\mathcal{E}}^n_{T, \pm}$ thanks to the continuous and dense inclusion)
\end{lemma}

\begin{proof}

Thanks to the lemma \ref{density2}, we already have that $(C^{\infty}_0(\R \times \mathbb{S}^2))^2 \cap \mathcal{E}^{fin,n}_{T, \pm}$ is dense in $\mathcal{E}^n_{T, \pm}$. Therefore, we only have to deal with the integral condition:
\begin{equation} \label{condition} \int_{-\infty}^{+\infty} e^{-i\int_{s}^{0}(l-l_+)}(u_1 + l_+ u_0)(s)\dd s = 0 \end{equation}
Let $u \in (C^{\infty}_0(\R \times \mathbb{S}^2))^2 \cap \mathcal{E}^{fin,n}_{T, \pm}$, we define $\psi \in C^{\infty}_0(\R_x)$ such that $\int_{-\infty}^{+\infty} e^{-i\int_{s}^{0}(l-l_+)}\psi(s)\dd s = 1$. We also define
\[ \psi_k(x) = \frac{1}{k}\psi\left( \frac{x}{k}\right) e^{i\int_{x}^{x/k}(l-l_+)} \]
By a change of variable $ t = \frac{s}{k}$, we get that
\[\int_{-\infty}^{+\infty} e^{-i\int_{s}^{0}(l-l_+)}\psi_k(s)\dd s = 1 \]
and with the same change of variable, we get
\[
\lim\limits_{k\to +\infty}\left\| \psi_k \right\|_{L^2(\R_x)} = 0
\]
So if we define $u^k = \begin{pmatrix} u_0 \\u_1 - \psi_k \int_{-\infty}^{+\infty} e^{-i\int_{s}^{0}(l-l_+)}(u_1 + l_+ u_0)(s)\dd s\end{pmatrix}$, we have $u^l \in (C^{\infty}_0(\R \times \mathbb{S}^2))^2 \cap \mathcal{E}^{fin,n,L}_{T, \pm}$ and 
\[
\lim\limits_{k \to +\infty} \left\| u - u_k \right\|_{\mathcal{E}^n_{T,+}} = 0
\]
\end{proof}

\begin{lemma}
\label{densityFin}
$(C^{\infty}_0( \R\times \mathbb{S}^2))^2\cap \mathcal{E}^{fin,n}_{+\infty}$ is dense in $\mathcal{E}^n_{+\infty}$ (hence in $\dot{\mathcal{E}}^n_{+\infty}$).
Also, for all $q\in\N$, we have the density of $(C^{\infty}_0 (\R \times \mathbb{S}^2))^2 \cap \mathcal{E}^{q,n}_{+\infty}$ in $Dom(H^n_{+\infty})\cap \mathcal{E}^{q,n}_{+\infty}$ equipped with the graph norm (hence in $Dom(\dot{H}^n_{+\infty})\cap \mathcal{E}^{q,n}_{+\infty}$ by using lemma 3.16 of \cite{GGH}).
\end{lemma}

\begin{proof}
The first part of the lemma is a reminder of the lemma \ref{density}. The density for $q$ fixed can be deduced from the proof. Indeed with the notations of the proof of lemma \ref{density}, if $u\in \mathcal{E}^{q,n}_{+\infty}$, the sequence converging to $u$ (in every norm considered) is of the form $\sum_{i=0}^N\Pi_{e_i}(u)$ and $\Pi_{e_i}(u)= 0$ if $e_i \notin Z_q$. So each term of
 the sequence is also in $\mathcal{E}^{q,n}_{+\infty}$
\end{proof}

\begin{lemma}
\label{densityA}
Let $\epsilon > 0$.
The set $\mathcal{A} = \left\{ \chi (\dot{H}_{+\infty})w^{-\epsilon} u, u \in \dot{\mathcal{E}}^{fin,n}_{+\infty}, \chi \in C^{\infty}_0 \right\}$ is dense in $\dot{\mathcal{E}}_{+\infty}$.
\end{lemma}

\begin{proof}
We already know, by lemma \ref{densityFin} that $\dot{\mathcal{E}}^{fin,n}_{+\infty}$ is dense in $\dot{\mathcal{E}}^n_{+\infty}$. We can approximate (up to an error of size $\eta>0$) $u \in \dot{\mathcal{E}}^{fin,n}_{+\infty}$ by $w^{-\epsilon}w^{\epsilon}\phi u$ where $\phi \in C^{\infty}_0(\mathbb{R}_x)$ is equal to $1$ on a sufficiently large ball and $w^{\epsilon}\phi u \in \dot{\mathcal{E}}^{fin,n}_{+\infty}$ (because $P$ commutes with $w^{\epsilon}\phi$. Then by functional calculus for selfadjoint operators, we can choose $\chi \in C^{\infty}_0(\mathbb{R})$ such that  $\chi(\dot{H}_{+\infty})w^{-\epsilon}w^{\epsilon}\phi u$ is an approximation of $u$ up to an error of size less than $2\eta$.
\end{proof}
\end{appendix}

\bibliography{biblio} 
\end{document}